\documentclass[12pt]{amsart}       
\usepackage{txfonts}
\usepackage{amssymb}
\usepackage{eucal}
\usepackage{graphicx}
\usepackage{amsmath}
\usepackage{amscd}
\usepackage[all]{xy}           

\usepackage{amsfonts,latexsym}
\usepackage{xspace}
\usepackage{epsfig}
\usepackage{float}
\usepackage{color}
\usepackage{fancybox}
\usepackage{colordvi}
\usepackage{multicol}
\usepackage{colordvi}
\usepackage{tikz}

\usepackage{wasysym}
\usepackage[active]{srcltx} 
\usepackage[colorlinks,final,backref=page,hyperindex,hypertex]{hyperref}

\newcommand{\nc}{\newcommand}
\newcommand{\delete}[1]{}

\nc{\mlabel}[1]{\label{#1}}  
\nc{\mcite}[1]{\cite{#1}}  
\nc{\mref}[1]{\ref{#1}}  
\nc{\mbibitem}[1]{\bibitem{#1}} 

\delete{
\nc{\mlabel}[1]{\label{#1}  
{\hfill \hspace{1cm}{\small\tt{{\ }\hfill(#1)}}}}
\nc{\mcite}[1]{\cite{#1}{\small{\tt{{\ }(#1)}}}}  
\nc{\mref}[1]{\ref{#1}{{\tt{{\ }(#1)}}}}  
\nc{\mbibitem}[1]{\bibitem[\bf #1]{#1}} 
}

\newtheorem{theorem}{Theorem}[section]
\newtheorem{prop}[theorem]{Proposition}
\newtheorem{lemma}[theorem]{Lemma}

\theoremstyle{definition}
\newtheorem{defn}[theorem]{Definition}
\newtheorem{prop-def}{Proposition-Definition}[section]

\newtheorem{remark}[theorem]{Remark}

\newtheorem{tempex}[theorem]{Example}
\newtheorem{tempexs}[theorem]{Examples}
\newtheorem{temprmk}[theorem]{Remark}
\newtheorem{tempexer}{Exercise}[section]

\nc{\vsa}{\vspace{-.1cm}} \nc{\vsb}{\vspace{-.2cm}}
\nc{\vsc}{\vspace{-.3cm}} \nc{\vsd}{\vspace{-.4cm}}
\nc{\vse}{\vspace{-.5cm}}

\nc{\Irr}{\mathrm{Irr}}
\nc{\ncrbw}{\calr}  
\nc{\NS}{U_{NS}}
\nc{\FN}{F_{\mathrm Nij}}
\nc{\dfgen}{V} \nc{\dfrel}{R}
\nc{\dfgenb}{\vec{v}} \nc{\dfrelb}{\vec{r}}
\nc{\dfgene}{v} \nc{\dfrele}{r}
\nc{\dfop}{\odot}
\nc{\dfoa}{\dfop^{(1)}} \nc{\dfob}{\dfop^{(2)}}
\nc{\dfoc}{\dfop^{(3)}} \nc{\dfod}{\dfop^{(4)}}
\nc{\mapm}[1]{\lfloor\!|{#1}|\!\rfloor}
\nc{\cmapm}[1]{\frakC(#1)}
\nc{\red}{\mathrm{Red}}
\nc{\cm}{C}
\nc{\supp}{\mathrm{Supp}}
\nc{\lex}{\mathrm{lex}}

\nc{\disp}[1]{\displaystyle{#1}}
\nc{\bin}[2]{ (_{\stackrel{\scs{#1}}{\scs{#2}}})}  
\nc{\binc}[2]{ \left (\!\! \begin{array}{c} \scs{#1}\\
    \scs{#2} \end{array}\!\! \right )}  
\nc{\bincc}[2]{  \left ( {\scs{#1} \atop
    \vspace{-.5cm}\scs{#2}} \right )}  
\nc{\sarray}[2]{\begin{array}{c}#1 \vspace{.1cm}\\ \hline
    \vspace{-.35cm} \\ #2 \end{array}}
\nc{\bs}{\bar{S}} \nc{\ep}{\epsilon}
\nc{\dbigcup}{\stackrel{\bullet}{\bigcup}}
\nc{\la}{\longrightarrow} \nc{\cprod}{\ast} \nc{\rar}{\rightarrow}
\nc{\dar}{\downarrow} \nc{\labeq}[1]{\stackrel{#1}{=}}
\nc{\dap}[1]{\downarrow \rlap{$\scriptstyle{#1}$}}
\nc{\uap}[1]{\uparrow \rlap{$\scriptstyle{#1}$}}
\nc{\defeq}{\stackrel{\rm def}{=}} \nc{\dis}[1]{\displaystyle{#1}}
\nc{\dotcup}{\ \displaystyle{\bigcup^\bullet}\ }
\nc{\sdotcup}{\tiny{ \displaystyle{\bigcup^\bullet}\ }}
\nc{\fe}{\'{e}}
\nc{\hcm}{\ \hat{,}\ } \nc{\hcirc}{\hat{\circ}}
\nc{\hts}{\hat{\shpr}} \nc{\lts}{\stackrel{\leftarrow}{\shpr}}
\nc{\denshpr}{\den{\shpr}}
\nc{\rts}{\stackrel{\rightarrow}{\shpr}} \nc{\lleft}{[}
\nc{\lright}{]} \nc{\uni}[1]{\tilde{#1}} \nc{\free}[1]{\bar{#1}}
\nc{\freea}[1]{\tilde{#1}} \nc{\freev}[1]{\hat{#1}}
\nc{\dt}[1]{\hat{#1}}
\nc{\wor}[1]{\check{#1}}
\nc{\intg}[1]{F_C(#1)}
\nc{\den}[1]{\check{#1}} \nc{\lrpa}{\wr} \nc{\mprod}{\pm}
\nc{\dprod}{\ast_P} \nc{\curlyl}{\left \{ \begin{array}{c} {} \\
{} \end{array}
    \right .  \!\!\!\!\!\!\!}
\nc{\curlyr}{ \!\!\!\!\!\!\!
    \left . \begin{array}{c} {} \\ {} \end{array}
    \right \} }
\nc{\longmid}{\left | \begin{array}{c} {} \\ {} \end{array}
    \right . \!\!\!\!\!\!\!}
\nc{\lin}{\call} \nc{\ot}{\otimes}
\nc{\ora}[1]{\stackrel{#1}{\rar}}
\nc{\ola}[1]{\stackrel{#1}{\la}}
\nc{\scs}[1]{\scriptstyle{#1}} \nc{\mrm}[1]{{\rm #1}}
\nc{\margin}[1]{\marginpar{\rm #1}}   
\nc{\dirlim}{\displaystyle{\lim_{\longrightarrow}}\,}
\nc{\invlim}{\displaystyle{\lim_{\longleftarrow}}\,}
\nc{\mvp}{\vspace{0.5cm}}
\nc{\mult}{m}       
\nc{\svp}{\vspace{2cm}} \nc{\vp}{\vspace{8cm}}
\nc{\proofbegin}{\noindent{\bf Proof: }}
\nc{\proofend}{$\blacksquare$ \vspace{0.5cm}}
\nc{\sha}{{\mbox{\cyr X}}}  
\nc{\ncsha}{{\mbox{\cyr X}^{\mathrm NC}}}
\newfont{\scyr}{wncyr10 scaled 550}
\nc{\ssha}{\mbox{\bf \scyr X}}
\nc{\ncshao}{{\mbox{\cyr X}^{\mathrm NC,\,0}}}
\nc{\shpr}{\diamond}    
\nc{\shprc}{\shpr_c}
\nc{\shpro}{\diamond^0}    
\nc{\shpru}{\check{\diamond}} \nc{\spr}{\cdot}
\nc{\catpr}{\diamond_l} \nc{\rcatpr}{\diamond_r}
\nc{\lapr}{\diamond_a} \nc{\lepr}{\diamond_e} \nc{\sprod}{\bullet}
\nc{\un}{u}                 
\nc{\vep}{\varepsilon} \nc{\labs}{\mid\!} \nc{\rabs}{\!\mid}
\nc{\hsha}{\widehat{\sha}} \nc{\psha}{\sha^{+}} \nc{\tsha}{\tilde{\sha}}
\nc{\lsha}{\stackrel{\leftarrow}{\sha}}
\nc{\rsha}{\stackrel{\rightarrow}{\sha}} \nc{\lc}{\lfloor}
\nc{\rc}{\rfloor} \nc{\sqmon}[1]{\langle #1\rangle}
\nc{\altx}{\Lambda} \nc{\vecT}{\vec{T}} \nc{\piword}{{\mathfrak P}}
\nc{\lbar}[1]{\overline{#1}}
\nc{\dep}{\mathrm{dep}}

\nc{\mmbox}[1]{\mbox{\ #1\ }}
\nc{\ayb}{\mrm{AYB}} \nc{\mayb}{\mrm{mAYB}} \nc{\cyb}{\mrm{cyb}}
\nc{\ann}{\mrm{ann}} \nc{\Aut}{\mrm{Aut}} \nc{\cabqr}{\mrm{CABQR
}} \nc{\can}{\mrm{can}} \nc{\colim}{\mrm{colim}}
\nc{\Cont}{\mrm{Cont}} \nc{\rchar}{\mrm{char}}
\nc{\cok}{\mrm{coker}} \nc{\dtf}{{R-{\rm tf}}} \nc{\dtor}{{R-{\rm
tor}}}

\nc{\Div}{{\mrm Div}} \nc{\End}{\mrm{End}} \nc{\Ext}{\mrm{Ext}}
\nc{\FG}{\mrm{FG}} \nc{\Fil}{\mrm{Fil}} \nc{\Frob}{\mrm{Frob}}
\nc{\Gal}{\mrm{Gal}} \nc{\GL}{\mrm{GL}} \nc{\Hom}{\mrm{Hom}}
\nc{\hsr}{\mrm{H}} \nc{\hpol}{\mrm{HP}} \nc{\id}{\mrm{id}} \nc{\Id}{\mathrm{Id}}  \nc{\ID}{\mathrm{ID}}
\nc{\im}{\mrm{im}} \nc{\incl}{\mrm{incl}} \nc{\Loday}{\mrm{ABQR}\
} \nc{\length}{\mrm{length}} \nc{\LR}{\mrm{LR}} \nc{\mchar}{\rm
char} \nc{\pmchar}{\partial\mchar} \nc{\map}{\mrm{Map}}
\nc{\MS}{\mrm{MS}} \nc{\OS}{\mrm{OS}} \nc{\NC}{\mrm{NC}}
\nc{\rba}{\rm{Rota-Baxter algebra}\xspace}
\nc{\rbas}{\rm{Rota-Baxter algebras}\xspace}
\nc{\rbw}{\ncrbw}
\nc{\rbws}{\rm{RBWs}\xspace}
\nc{\rbadj}{\rm{RB}\xspace}
\nc{\mpart}{\mrm{part}} \nc{\ql}{{\QQ_\ell}} \nc{\qp}{{\QQ_p}}
\nc{\rank}{\mrm{rank}} \nc{\rcot}{\mrm{cot}} \nc{\rdef}{\mrm{def}}
\nc{\rdiv}{{\rm div}} \nc{\rtf}{{\rm tf}} \nc{\rtor}{{\rm tor}}
\nc{\res}{\mrm{res}} \nc{\SL}{\mrm{SL}} \nc{\Spec}{\mrm{Spec}}
\nc{\tor}{\mrm{tor}} \nc{\Tr}{\mrm{Tr}}
\nc{\mtr}{\mrm{tr}}

\nc{\ab}{\mathbf{Ab}} \nc{\Alg}{\mathbf{Alg}}
\nc{\Bax}{\mathbf{CRB}} \nc{\Algo}{\mathbf{Alg}^0}
\nc{\cRB}{\mathbf{CRB}} \nc{\cRBo}{\mathbf{CRB}^0}
\nc{\RBo}{\mathbf{RB}^0} \nc{\BRB}{\mathbf{RB}}
\nc{\Dend}{\mathbf{DD}} \nc{\bfk}{{\bf k}} \nc{\bfone}{{\bf 1}}
\nc{\base}[1]{{a_{#1}}} \nc{\Cat}{\mathbf{Cat}}
 \nc{\DN}{\mathbf{DN}}
\nc{\NA}{\mathbf{NA}}
\nc{\SDN}{\mathbf{SDN}}
\nc{\Diff}{\mathbf{Diff}} \nc{\gap}{\marginpar{\bf
Incomplete}\noindent{\bf Incomplete!!}
    \svp}
\nc{\FMod}{\mathbf{FMod}} \nc{\Int}{\mathbf{Int}}
\nc{\Mon}{\mathbf{Mon}}
\nc{\RB}{\mathbf{RB}} \nc{\remarks}{\noindent{\bf Remarks: }}
\nc{\Rep}{\mathbf{Rep}} \nc{\Rings}{\mathbf{Rings}}
\nc{\Sets}{\mathbf{Sets}} \nc{\bfx}{\mathbf{x}}
\nc{\BA}{{\Bbb A}} \nc{\CC}{{\Bbb C}} \nc{\DD}{{\Bbb D}}
\nc{\EE}{{\Bbb E}} \nc{\FF}{{\Bbb F}} \nc{\GG}{{\Bbb G}}
\nc{\HH}{{\Bbb H}} \nc{\LL}{{\Bbb L}} \nc{\NN}{{\Bbb N}}
\nc{\QQ}{{\Bbb Q}} \nc{\RR}{{\Bbb R}} \nc{\TT}{{\Bbb T}}
\nc{\VV}{{\Bbb V}} \nc{\ZZ}{{\Bbb Z}}

\nc{\cala}{{\mathcal A}} \nc{\calb}{{\mathcal B}}
\nc{\calc}{{\mathcal C}}
\nc{\cald}{{\mathcal D}} \nc{\cale}{{\mathcal E}}
\nc{\calf}{{\mathcal F}} \nc{\calg}{{\mathcal G}}
\nc{\calh}{{\mathcal H}} \nc{\cali}{{\mathcal I}}
\nc{\calj}{{\mathcal J}} \nc{\call}{{\mathcal L}}
\nc{\calm}{{\mathcal M}} \nc{\caln}{{\mathcal N}}
\nc{\calo}{{\mathcal O}} \nc{\calp}{{\mathcal P}}
\nc{\calr}{{\mathcal R}} \nc{\cals}{{\mathcal S}} \nc{\calt}{{\mathcal T}}
\nc{\calw}{{\mathcal W}} \nc{\calx}{{\mathcal X}} \nc{\caly}{{\mathcal Y}} \nc{\calz}{{\mathcal Z}}
\nc{\CA}{\mathcal{A}}

\nc{\frakA}{{\mathfrak A}}
\nc{\fraka}{{\mathfrak a}}
\nc{\frakB}{{\mathfrak B}}
\nc{\frakb}{{\mathfrak b}}
\nc{\frakC}{{\mathfrak C}}
\nc{\frakd}{{\mathfrak d}}
\nc{\frakF}{{\mathfrak F}}
\nc{\frakg}{{\mathfrak g}}
\nc{\frakm}{{\mathfrak m}}
\nc{\frakM}{{\mathfrak M}}
\nc{\frakMo}{{\mathfrak M}^0}
\nc{\frakP}{{\mathfrak P}}
\nc{\frakp}{{\mathfrak p}}
\nc{\frakS}{{\mathfrak S}}
\nc{\frakSo}{{\mathfrak S}^0}
\nc{\fraks}{{\mathfrak s}}
\nc{\os}{\overline{\fraks}}
\nc{\frakT}{{\mathfrak T}}
\nc{\frakTo}{{\mathfrak T}^0}
\nc{\oT}{\overline{T}}
\nc{\frakX}{{\mathfrak X}}
\nc{\frakXo}{{\mathfrak X}^0}
\nc{\frakx}{{\mathbf x}}
\nc{\frakTx}{\frakT}      
\nc{\frakTa}{\frakT^a}        
\nc{\frakTxo}{\frakTx^0}   
\nc{\caltao}{\calt^{a,0}}   
\nc{\ox}{\overline{\frakx}} \nc{\fraky}{{\mathfrak y}}
\nc{\frakz}{{\mathfrak z}} \nc{\oX}{\overline{X}} \font\cyr=wncyr10

\nc{\tred}[1]{\textcolor{red}{#1}} \nc{\tgreen}[1]{\textcolor{green}{#1}}
\nc{\tblue}[1]{\textcolor{blue}{#1}} \nc{\tpurple}[1]{\textcolor{purple}{#1}}

\nc{\li}[1]{\tpurple{\underline{Li:}#1 }}
\nc{\liadd}[1]{\tpurple{#1}}
\nc{\xing}[1]{\tblue{\underline{Xing:}#1 }}

\nc{\markus}[1]{\tred{\underline{Markus:} #1}}

\nc{\dnx}{\Delta_n X} \nc{\dx}{\Delta X} \nc{\dgp}{{\rm deg_{P}}}
\nc{\dgt}{{\rm deg_{T}}} \nc{\dg}{{\rm deg}} \nc{\ida}{ID($A$)} \nc{\tu}{\tilde{u}} \nc{\tv}{\tilde{v}}
\nc{\nr}{\calr_n} \nc{\nz}{\calz_n} \nc{\fun}{\cala_{n,d}}
 \nc{\fbase}{\calb} \nc{\LF}{\mathrm{RF}} \nc{\FFA}{\mathrm{LF}} \nc{\irr}{\mathrm{Irr}}
 \nc{\result}{\bfk\mathrm{Irr}(S_n)}  \nc{\I}{I_{\mathrm{ID},n}^0}
 \nc{\nrs}{\calr_n^\star} \nc{\ii}{\mathrm{I}} \nc{\iii}{\mathrm{II}}

\nc{\ws}[1]{{#1}}
\nc{\deleted}[1]{\delete{#1}}
\nc{\plas}{placements\xspace}

\nc{\bvp}[2]{\boxed{\begin{array}{l}#1\\#2\end{array}}}
\nc{\evl}{E}
\nc{\cum}{{\textstyle \varint}}

\begin{document}

\title[Free integro-differential algebras]{Free integro-differential algebras and Gr\"{o}bner-Shirshov bases}

\author{Xing Gao}
\address{School of Mathematics and Statistics,
Key Laboratory of Applied Mathematics and Complex Systems,
Lanzhou University, Lanzhou, Gansu, 730000, P.R. China}
\email{gaoxing@lzu.edu.cn}
\author{Li Guo}
\address{
Department of Mathematics and Computer Science,
Rutgers University,
Newark, NJ 07102, USA}
\email{liguo@rutgers.edu}

\author{Markus Rosenkranz}
\address{
 	School of Mathematics, Statistics and Actuarial Science,
	University of Kent,
	Canterbury CT2 7NF, England}
\email{M.Rosenkranz@kent.ac.uk}

\date{\today}

\begin{abstract}
  The notion of commutative integro-differential algebra was
  introduced for the algebraic study of boundary problems for linear
  ordinary differential equations. Its noncommutative analog achieves
  a similar purpose for linear systems of such equations. In both
  cases, free objects are crucial for analyzing the underlying
  algebraic structures, e.g.\@ of the (matrix) functions.

  In this paper we apply the method of Gr\"obner-Shirshov bases to
  construct the free (noncommutative) integro-differential algebra on
  a set. The construction is from the free Rota-Baxter algebra on the
  free differential algebra on the set modulo the differential
  Rota-Baxter ideal generated by the noncommutative integration by
  parts formula. In order to obtain a canonical basis for this quotient, we first reduce to the case when the set is
  finite. Then in order to obtain the monomial order needed for the
  Composition-Diamond Lemma, we consider the free Rota-Baxter algebra on the truncated free differential algebra. A
  Composition-Diamond Lemma is proved in this context, and a
  Gr\"obner-Shirshov basis is found for the corresponding differential Rota-Baxter ideal.
\end{abstract}

\keywords{
Integro-differential algebra, free objects, Gr\"obner-Shirshov bases, Rota-Baxter algebra, differential Rota-Baxter algebra.
}

\maketitle

\tableofcontents

\setcounter{section}{0}

\allowdisplaybreaks

\section{Introduction}

\subsection{Commutative Setting}

An \emph{integro-differential algebra}~$(R,d,P)$ is an algebraic abstraction of
the familiar setting of calculus, where one employs a notion of
differentiation~$d$ together with a notion of integration~$P$ on some (real or
complex) algebra of functions.

For understanding the motivation behind this abstraction, let us first consider
the~$(R,d)$. This is the familiar setting of \emph{differential algebra} as set
up in the work of Ritt~\cite{Ri1,Ri2} and Kolchin~\cite{Ko}. The idea is to
capture the structure of (polynomially) nonlinear differential equations from a
purely algebraic viewpoint. If one speaks of solutions in this context, one
usually means elements in a suitable differential field~$\bar{R}$
extending~$R$. In particular, in differential Galois theory, an ``integral''
of~$f \in R$ is taken as an element~$u \in \bar{R}$ such that~$d(u) = f$.

In applications, however, differential equations often come together with
\emph{boundary conditions} (for simplicity here we include also initial
conditions under this term). Incorporating these into the algebraic model
requires some modifications: Assuming every~$f \in R$ has an integral~$u \in R$,
the condition~$d(u) = f$ becomes~$d \circ P = 1_R$, and it is natural to assume
that the operator~$P\colon f \mapsto u$ is linear. In the standard setting~$R =
C^\infty(\RR)$ we have~$d(u) = u'$ and~$P(f) = \cum_a^x f(\xi) \, d\xi$ for some
initial point~$a \in \RR$. This leads us to expect some further properties
of~$P$:
\begin{itemize}
\item The Fundamental Theorem of Calculus tells us that~$P$ is a right inverse
  of~$d$, as noted above. But it also tells us that~$P$ is \emph{not} a left
  inverse; rather, we have~$P \circ d = 1_R - \evl_a$ in the standard setting,
  where~$\evl_a$ is the \emph{evaluation}~$u \mapsto u(a)$. Note that~$\evl_a$ is
  a multiplicative functional on~$R$.
\item Just like~$d$ satisfies the product rule (also known as the Leibniz law),
  so $P$ satisfies the well-known \emph{integration by parts} rule. In its
  strong form, this is the rule~$P(fd(g)) = fg - P(d(f)g) - \evl(f)\evl(g)$; in
  its weak form it is given by~$P(f)P(g) = P(f P(g)) + P(P(f) g)$. Both can be
  verified immediately in the standard setting; for their distinction in general
  see below.
\end{itemize}
We will now explain briefly why both of these properties are instrumental for
treating \emph{boundary problems} (differential equations with boundary
conditions) on an algebraic level. We restrict ourselves to the classical case
of two-point boundary problems for a linear ordinary differential equations. For
this and the more general setting of Stieltjes boundary conditions, we refer
to~\cite{RR}.

If~$R$ is an arbitrary~$\bfk$-algebra, we can define an \emph{evaluation} as a
multiplicative linear functional~$R \to \bfk$. In the case of a two-point
boundary problem over~$[a,b] \subset \RR$, one will have two
evaluations~$\evl_a\colon u \mapsto u(a)$ and~$\evl_b\colon u \mapsto u(b)$. A
boundary condition like~$2u(a) - 3u'(a) + u'(b) = 0$ then translates
to~$\beta(u) = 0$ with the linear functional~$\beta = 2\evl_a - 3 \evl_a d +
\evl_b d$.

We can now define a general boundary problem over~$(R,d, \evl_a, \evl_b)$ as the
task of finding for given~$f \in R$ the solution $u \in R$ of
\begin{equation*}
  \label{eq:bvp}
  \bvp{Tu=f,}{\beta_1(u) = \cdots = \beta_n(u) = 0,}
\end{equation*}
where~$T \in R[d]$ is a monic linear differential operator of order~$n$ and the
boundary conditions~$\beta_i$ are linear functionals built from~$d$ and the
evaluations~$\evl_a, \evl_b$ as above, with differentiation order below~$n$. We
call the boundary problem~\eqref{eq:bvp} \emph{regular} if there is a unique
solution~$u \in R$ for every~$f \in R$. In this case, the association~$f \mapsto
u$ gives rise to linear map~$G\colon R \to R$ known as the \emph{Green's
  operator} of~\eqref{eq:bvp}.

It turns out~\cite[Thm.~26]{RR} that the Green's operator~$G$ of~\eqref{eq:bvp}
can be computed algebraically from a given fundamental system of~$T$. Moreover,
$G$ can be written in the form of an integral operator~$u = \cum_a^b g(x,\xi) \,
f(\xi) \, d\xi$, where~$g(x,\xi)$ is the so-called \emph{Green's function}
of~\eqref{eq:bvp}. More precisely, defining the operator ring generated
by~$R[d]$, the integral operator~$P$ and the evaluations~$\evl_a, \evl_b$,
modulo suitable relations, $G$ can be written as an element of this quotient
ring, with~$g$ as its canonical representative. We observe that a \emph{single}
integration is sufficient for undoing~$n$ differentiations---this is achieved by
collapsing~$n$ integrations into one, using integration by parts as one of the
relations.

In fact, the relations contain two different rules that encode \emph{integration
  by parts}: The rewrite rule~$\cum f \cum \to \dots$ encapsulates the weak
form~$P(f)P(g) = P(f P(g)) + P(P(f) g)$ while the rewrite rule~$\cum f \partial
\to \dots$ encodes the strong form~$P(fd(g)) = fg - P(d(f)g) -
\evl(f)\evl(g)$. The former contracts multiple integrations into one, the
purpose of latter is to eliminate derivatives from the Green's operator.

In concluding this brief account on the algebraic treatment of boundary
problems, let us note that the operator ring is much more general than the usual
Green's functions. Extending two-point conditions to \emph{Stieltjes boundary
  conditions} leads to a threefold generalization: More than two point
evaluations can be used, definite integrals may appear, and the differentiation
order need not be lower than that of~$T$. In this case, $G$ is still
representable as an element of the operator ring, and as before it may be
computed from a given fundamental system of~$T$.

Let us now turn to the distinction between the ``weak'' form (also called
Rota-Baxter axiom) and the ``strong'' form (called the hybrid Rota-Baxter axiom)
of integration by parts. Since the former does not involve the derivation~$d$,
it can be used to encode an algebraic structure~$(R,P)$ with just an
integral---this leads to the important notion of a Rota-Baxter algebra,
introduced below in a more general context in
Def.~\ref{def:cats}\ref{it:rbalg}. Rota-Baxter algebras form an extremely rich
structure with important applications in combinatorics, physics (Yang-Baxter
equation, renormalization theory), and probability; see~\cite{Gub} for a
detailed survey. Here we restrict our interest to the interaction between the
Rota-Baxter operator~$P$ and the derivation~$d$. If this interaction is only
given by the section axiom~$d \circ P = 1_R$, one speaks of a \emph{differential
  Rota-Baxter algebra}, introduced formally in
Def.~\ref{def:cats}\ref{it:drbalg} below. Intuitively, this is a weak coupling
between the differential algebra~$(R,d)$ and the Rota-Baxter algebra~$(R,P)$.

In contrast, the hybrid Rota-Baxter axiom involves~$P$ as well as~$d$, and it
creates a stronger coupling between~$d$ and~$P$. In fact, one checks immediately
that it implies the Rota-Baxter axiom, but the converse is not in general true
as one sees from Example~3 in~\cite{RR}. An \emph{integro-differential
  algebra}~$(R,d,P)$ is then defined as a differential ring~$(R,d)$ with a right
inverse~$P$ of~$d$ that satisfies the hybrid Rota-Baxter axiom; see
Def.~\ref{def:cats}\ref{it:intdiffalg} for the more general setting. Hence every
integro-differential algebra is also a differential Rota-Baxter algebra but
generally not vice versa. The crucial difference between the two categories can
be expressed in various equivalent ways~\cite[Thm.~2.5]{GRR} of which we shall
mention only two. An integro-differential algebra~$(R,d,P)$ is a differential
Rota-Baxter algebra satisfying one of the following equivalent extra conditions:
\begin{itemize}
\item The projector~$E := 1_R - P \circ d$ is \emph{multiplicative}. So if
  additionally~$\ker{d} = \bfk$ as is typically the case in an ordinary
  differential algebra, then~$E$ deserves to be called an ``evaluation''. This
  is the situation we had observed before in the standard setting.
\item The image~$P(R)$ is not only a subalgebra (as in any Rota-Baxter algebra)
  but an \emph{ideal} of~$R$. As a consequence, this excludes the possibility
  that~$(R,d)$ has the structure of a differential field so common in
  differential Galois theory (see above).
\end{itemize}

In many ``natural'' examples---such as the standard setting described
above---the notions of differential Rota-Baxter algebra and integro-differential
algebra actually coincide. However, their differences are borne out fully when
it comes to constructing the corresponding \emph{free objects}: For differential
Rota-Baxter algebras, this works in the same way as for the free Rota-Baxter
algebra (only with differential instead of plain monomials). Due to the tighter
differential/Rota-Baxter coupling, the construction of the free
integro-differential algebra is significantly more complex. Two different
methods have been used to this end: In~\cite{GRR} an artificial evaluation is
set up while in~\cite{GGZ} Gr\"obner-Shirshov bases are employed.

Free objects are useful in many ways. In the case of the free
integro-differential algebra, we mention the following two \emph{applications},
where we think of the~$R$ as function spaces similar to the standard setting:
\begin{itemize}
\item It allows to build up integro-differential subalgebras~$R \subset
  C^\infty(\RR)$ by \emph{adjoining} new functions. For example, we can create
  the subalgebra of exponentials~$R = \RR[e^x]$ by forming the free
  integro-differential algebra in one indeterminate~$e$ and passing to the
  quotient modulo the integro-differential ideal generated by~$P(e) - e +
  1$. Note that this implies the differential relation~$d(e) = e$ and the
  initial value~$\evl(e) = 1$.
\item It attaches a rigorous meaning to the intuitive notion of \emph{purely
    algebraic manipulations of integro(-differential) equations}. For example,
  in the proof of the Picard-Lindel\"of theorem, one transforms a given initial
  value problem for a differential equation into an equivalent integral
  equation.
\end{itemize}
Intuitively, one should think of the elements in a free integro-differential as
an integro-differential generalization of differential polynomials (with trivial
derivation on the coefficients).

\subsection{Noncommutative Setting.}

Up to now we have thought of the ring~$R$ as commutative but the above
considerations---in particular the applications of the free integro-differential
algebra---will also make sense without the assumption of commutativity. In fact,
the noncommutative standard example is the (real or complex) \emph{matrix
  algebra}~$R = C^\infty(\RR)^{n \times n}$, and this forms the basis for
two-point (and more general) boundary problems for linear systems of ordinary
differential equations. Hence we may think of the (noncommutative) free object
as the substrate for adjoining matrix functions and manipulating systems of
integro-differential equations (the usual situation of the Picard-Lindel\"of
theorem).

This can immediately be generalized. The \emph{matrix functor} assigns
to an arbitrary (commutative or noncommutative) integro-differential
algebra~$(R,d,P)$ the (necessarily noncommutative)
integro-differential algebra~$(R^{n \times n}, \bar{d}, \bar{P})$
whose derivation~$\bar{d}$ and Rota-Baxter operator~$\bar{P}$ are
defined coordinatewise; the same is true for the transport of
morphisms from~$R \to S$ to~$R^{n \times n} \to S^{n \times n}$.

Another familiar functor from the category of integro-differential algebras to
itself is given by the construction of \emph{noncommutative
  polynomials}~$R\langle x_1, \dots, x_k \rangle$ over a commutative
integro-differential algebra~$(R,d,P)$, where the~$x_1, \dots, x_k$ are assumed
to commute with the coefficients in~$R$ but not amongst themselves. The
derivation and Rota-Baxter operator, as well as the transport of morphisms, are
defined coefficientwise.

The construction of~$R\langle x_1, \dots, x_k\rangle$ models some extensions of
a commutative integro-differential algebra to a larger noncommutative one: In
some cases, the larger algebra will be a quotient of~$R\langle x_1, \dots,
x_k\rangle$. A typical case is given by extending~$R = C^\infty(\RR)$
to~$R[i,j,k] := R\langle i, j, k\rangle/I$ where~$I$ is the ideal generated by
the familiar relations~$i^2 = j^2 = k^2 = -1$ and~$ij=k, jk=i, ki=j$ with their
anticommutative counterparts. Obviously~$R[i,j,k]$ can be seen as an algebraic
model for smooth \emph{quaternion-valued functions} of a real variable. (Finding
the right notions of differentiation and integration for functions of a
quaternion variable is a far more delicate process, giving rise to the
\emph{quaternion calculus}~\cite{Deav}. It would be interesting to investigate
this in the frame of noncommutative integro-differential algebras but this is
beyond the scope of the current paper.)

Finally, let us mention a potential application in combinatorics: In
\emph{species theory}~\cite{BLL}, the usage of derivations and so-called
combinatorial differential equations~\cite{L} is
well-established. Algebraically, the isomorphism classes of species form a
differential semiring that can be extended to a differential ring by introducing
so-called virtual species. Using the more restricted setting of linear species,
it is also possible to introduce an integral operator~\cite{BLL,PSS}, thus
endowing the class of virtual linear species with the structure of an
integro-differential ring. Since species can be extended to a noncommutative
setting~\cite{DP}, it would be interesting to see how an integro-differential
structure can be set up in this case.

\subsection{Structure of the Paper.}

In this paper we construct free integro-differential algebras. This
construction, built on an earlier construction of free differential
Rota-Baxter algebras~\mcite{GK3}, is obtained by applying the method
of Gr\"obner bases or Gr\"obner-Shirshov bases. The method has its
origin in the works of Buchberger~\mcite{Bu}, Hironaka~\mcite{Hi},
Shirshov~\mcite{Sh} and Zhukov~\mcite{Zh}. Even though it has been
fundamental for many years in commutative algebra, associative
algebra, algebraic geometry and computational
algebra~\mcite{Be,Bo}. It has only recently shown how comprehensive
the method of Gr\"obner-Shirshov bases can be, through the large
number of algebraic structures that the method has been successfully
applied to.  See~\mcite{BC,BCC,BCL,BLSZ} for further details.
The method is especially useful in constructing free objects in
various categories, including the alternative constructions of free
Rota-Baxter algebras and free differential Rota-Baxter
algebras~\mcite{BCD,BCQ}.
In the recent paper~\mcite{GGZ}, this method is applied to construct the free commutative integro-differential algebras.

The layout of the paper is as follows. In
\emph{Section~\mref{sec:ida}}, we give the definition of
integro-differential algebra and summarize the construction of free
differential Rota-Baxter algebras as a preparation for the
construction of free (noncommutative) integro-differential
algebras. In \emph{Section~\mref{sec:monorder}}, we set up a weakly
monomial order on differential Rota-Baxter monomials of order $n$. In
\emph{Section~\mref{sec:cd}}, we prove the Composition-Diamond Lemma
for free differential Rota-Baxter algebras of order $n$. In
\emph{Section~\mref{sec:gs}}, we prove that the differential
Rota-Baxter ideal of the free differential Rota-Baxter algebra that
defines the relations for free integro-differential algebras possesses
a Gr\"obner-Shirshov basis. Therefore we can apply the
Composition-Diamond Lemma to obtain a canonical basis, identified as
the set of functional monomials, for the free integro-differential
algebra of order $n$. We then show that the order $n$ pieces form a
direct system whose functional monomials accumulate to a canonical
basis of the free integro-differential algebra on a finite set
$X$. Finally, we prove that for an arbitrary set $X$, the inclusions
of the finite subsets of $X$ into $X$ also preserve the functional
monomials, which allows us to take their union as a canonical basis of
the free integro-differential algebra on $X$.

\section{Free integro-differential algebras}
\mlabel{sec:ida}

We recall the concepts of algebras with various differential and integral operators that lead to the integro-differential algebra. We also summarize the constructions of the free objects in the corresponding categories. See~\mcite{GG,GRR} for further details and examples.

\subsection{The definitions}
Algebras considered in this paper are assumed to be unitary, unless specified otherwise.

\begin{defn}\label{def:cats}
{\rm Let $\bfk$ be a unitary commutative ring. Let $\lambda\in \bfk$ be fixed.
\begin{enumerate}
\item\label{it:dalg} A {\bf differential $\bfk$-algebra of weight
    $\lambda$} (also called a {\bf $\lambda$-differential
    $\bfk$-algebra}) is defined to be an associative $\bfk$-algebra $R$ together
  with a linear operator $d \colon R\to R$ such that
\begin{equation}
d(1)=0,\ d(uv)=d(u)v+ud(v)+\lambda d(u)d(v) \text{ for all } u, v\in R.
\mlabel{eq:diffl}
\end{equation}

\item\label{it:rbalg} A {\bf Rota-Baxter $\bfk$-algebra of weight $\lambda$} is defined to be an
    associative $\bfk$-algebra $R$ together with a linear operator
    $P\colon R\to R$ such that
\begin{equation}
P(u)P(v)=P(uP(v))+P(P(u)v)+\lambda P(uv) \text{ for all } u, v\in R.\mlabel{eq:rb}
\end{equation}
\item\label{it:drbalg} A {\bf differential Rota-Baxter k-algebra of weight $\lambda$} (also called a {\bf $\lambda$-differential
      Rota-Baxter $\bfk$-algebra}) is defined to be a differential $\bfk$-algebra
    $(R,d)$ of weight $\lambda$ and a Rota-Baxter operator $P$ of
    weight $\lambda$ such that
\begin{equation} d\circ P=\id.
\mlabel{eq:fft}
\end{equation}
\item\label{it:intdiffalg} An {\bf integro-differential $\bfk$-algebra of weight
      $\lambda$} (also called a {\bf $\lambda$-integro-differential
      $\bfk$-algebra}) is defined to be a differential $\bfk$-algebra $(R,d)$ of
    weight $\lambda$ with a linear operator $P\colon R \to R$ that satisfies Eq.~(\mref{eq:fft}) and such that
\begin{equation}
\begin{aligned}
P(d(u)P(v))&= uP(v)-P(uv) - \lambda P(d(u)v) \text{ for all } u, v\in R,\\
P(P(u)d(v))&= P(u)v-P(uv) - \lambda P(ud(v)) \text{ for all } u, v\in R.
\mlabel{eq:ibpl}
\end{aligned}
\end{equation}
\end{enumerate}
}
\end{defn}

Eqs.~($\mref{eq:rb}$), ($\mref{eq:fft}$) and ($\mref{eq:ibpl}$) are called the {\bf Rota-Baxter axiom}, {\bf section axiom} and {\bf integration by parts axiom}, respectively. See~\mcite{GRR} for the equivalent conditions for the integration by parts axiom in various forms.

\subsection{Free differential algebras}
We recall the standard construction of free differential algebras. We also  introduce the concept of a differential polynomial algebra with bounded order as it will be needed later in the paper.

For a set $Y$, let $M(Y)$ be the free monoid on $Y$ with identity 1, and let $S(Y)$ be the free semigroup on $Y$. Thus elements in $M(Y)$ are words, plus the identity $1$, from the alphabet set $Y$. Further the noncommutative polynomial algebra $\bfk\langle Y\rangle$ on $Y$ is the semigroup algebra $\bfk M(Y)$.

\begin{theorem}
\begin{enumerate}
\item
Let $Y$ be a set with a map $d_0\colon Y\to Y$. Extend $d_0$ to $d \colon
\bfk\langle Y \rangle  \to \bfk\langle Y \rangle$ as follows. Let $w=u_1\cdots u_k,
u_i\in Y$, $1\leq i\leq k$, be a word from the alphabet set $Y$. Recursively define
\begin{equation}
      d(w)=d_0(u_1)u_2 \cdots u_k + u_1 d (u_2 \cdots u_k) + \lambda
      d_0(u_1)d(u_2 \cdots u_k).
      \mlabel{eq:prodind}
\end{equation}
Explicitly, we have
\begin{equation}
    d(w)=\sum_{\emptyset \neq I\subseteq [k]} \lambda^{|I|-1} d_I(u_1)\cdots d_I(u_k), \quad
    d_I(u_i):=d_{w,I}(u_i)=\left\{\begin{array}{ll} d(u_i), & i\in I, \\
    u_i, & i\not\in I. \end{array} \right.
    \mlabel{eq:prodexp}
\end{equation}
Further define $d(1)=0$ and then extend $d$ to $\bfk\langle Y \rangle$
by linearity. Then $(\bfk\langle Y \rangle,d)$ is a differential
algebra of weight $\lambda$.
\mlabel{it:diffs}
\item
Let $X$ be a set. Let $ Y:=\Delta X: = \{ x^{(n)} \mid
x\in X, n\geq 0\}$ with the map $d_0\colon \Delta X\to \Delta X, x^{(n)}\mapsto x^{(n+1)}$. Then with the extension $d$ of $d_0$ as in Eq.~(\mref{eq:prodind}), $(\bfk \langle \Delta X\rangle, d)$ is the free differential algebra of weight $\lambda$ on the set $X$.  \mlabel{it:commfreediff}
\item
For a given $n\geq 1$, let $\Delta X^{(n+1)}:=\left\{x^{(k)}\,\big|\, x\in X, k\geq n+1\right\}$. Then $\bfk\langle\dx \rangle\Delta X^{(n+1)}\bfk\langle\dx \rangle$ is the differential ideal $I_n$ of $\bfk\langle\dx \rangle$ generated by the set $\{ x^{(n+1)}\,|\,x\in X\}$. The quotient differential algebra $\bfk\langle \dx\rangle/I_n$ is of order $n$ and has a canonical basis given by
$$\Delta_n X:=\{x^{(k)}\,|\, x\in X, k\leq n\},$$
thus giving a differential algebra isomorphism $\bfk\langle \dx \rangle/I_n\cong \bfk\langle\Delta_n X \rangle$, called the {\bf differential polynomial algebra of order $n$}. Here the differential structure on the later algebra is given by
\begin{equation*}
d(x^{(i)}) =
\left\{ \begin{array}{ll} x^{(i+1)}, & 1\leq i\leq n-1, \\ 0, & i=n. \end{array}\right .
 \mlabel{eq:diffn}
\end{equation*}
\mlabel{it:diffordn}
\end{enumerate}
\mlabel{thm:diff}
\end{theorem}
\begin{proof} Item~\mref{it:diffs} is a generalization of Item~\mref{it:commfreediff} from~\mcite{GK3} and can be proved in the same way. Item~\mref{it:diffordn} is a direct consequence of Item~\mref{it:commfreediff}.
\end{proof}

\subsection{Free operated algebras}

We now recall the construction of the free operated algebra on a set $X$ that has appeared in various studies. In particular it gives the free (differential) Rota-Baxter algebra as a quotient~\mcite{BCQ,Gop,Gub,GSZ}. 

\begin{defn}
{\rm An {\bf operated monoid (resp. $\bfk$-algebra) with operator set $\Omega$} is defined to be a monoid (resp. $\bfk$-algebra) $G$ together with a set of maps $\alpha_\omega\colon G\to G,\omega\in \Omega$. A morphism between operated monoids (resp. $\bfk$-algebras) $(G,\{\alpha_\omega\}_\omega)$ and $(H,\{\beta_\omega\}_\omega)$ is a monoid (resp. $\bfk$-algebra) homomorphism $f\colon G\to H$ such that $f\circ \alpha_\omega=\beta_\omega\circ f$ for $\omega\in \Omega$.
}
\end{defn}
We next construct the free operated monoids generated by a set.

Fix a set $Y$. We define monoids $\frakM_{\Omega,n}:=\frakM_{\Omega,n}(Y)$ for $n\geq 0$ by the following recursion. We use the notation $\sqcup$ for disjoint union.

First denote $\frakM_{\Omega,0}:= M(Y)$.
Let $\lc M(Y)\rc_\omega:=\{\lc u\rc_\omega\,|\, u\in M(Y)\}, \omega\in \Omega,$ be disjoint sets in bijection with and disjoint from $M(Y)$.
Then define
$$\frakM_{\Omega,1}:= M(Y\sqcup (\sqcup_{\omega\in \Omega} \lc M(Y)\rc_\omega)).$$
Even though elements in $\lc M(Y)\rc_\omega$ are symbols indexed by elements in $M(Y)$, the sets $\lc M(Y)\rc_\omega$ and $M(Y)$ are disjoint. In particular $\lc1\rc_\omega$ is a symbol that is different from $1$.

The natural inclusion $Y\hookrightarrow Y\sqcup (\sqcup_{\omega\in \Omega}\lc\frakM_{\Omega,0}\rc_\omega)$ induces a monomorphism
$i_{0,1}\colon  \frakM_{\Omega,0}=M(Y)\hookrightarrow \frakM_{\Omega,1}=M(Y\sqcup (\sqcup_{\omega \in \Omega}\, \lc\frakM_{\Omega,0}\rc_\omega))$ of free monoids, allowing we to identify $\frakM_{\Omega,0}$ with its image in $\frakM_{\Omega,1}$.
Assume that $\frakM_{\Omega,m-1}$ has been defined for $m\geq 2$ and that the embedding
\begin{equation}
i_{m-2,m-1}\colon  \frakM_{\Omega,m-2} \hookrightarrow \frakM_{\Omega,m-1}
\mlabel{eq:emb}
\end{equation}
has been obtained. We define
\begin{equation*}
 \frakM_{\Omega,m}:=M(Y\sqcup (\sqcup_{\omega \in \Omega} \lc\frakM_{\Omega,m-1}\rc_\omega) ).
 \mlabel{eq:frakm}
 \end{equation*}
From the embedding in Eq.~(\mref{eq:emb}), we obtain the injection
$$  \lc\frakM_{\Omega,m-2}\rc_\omega \hookrightarrow
    \lc \frakM_{\Omega,m-1} \rc_\omega, \ \omega\in \Omega.$$
Thus by the universal property of $\frakM_{\Omega,m-1}=M(Y\sqcup (\sqcup_{\omega \in \Omega} \lc\frakM_{\Omega,m-2}\rc_\omega))$ as a free monoid, we have
\begin{eqnarray*}
\frakM_{\Omega,m-1} &=& M(Y\sqcup (\sqcup_{\omega \in \Omega} \lc\frakM_{\Omega,m-2}\rc_\omega))\hookrightarrow
    M(Y\sqcup (\sqcup_{\omega \in \Omega} \lc \frakM_{\Omega,m-1}\rc_\omega)) =\frakM_{\Omega,m}.
\end{eqnarray*}
This completes the inductive construction of the monoids $\frakM_{\Omega,n}, n\geq 0$. 

We finally define the monoid from the direct limit
$$ \frakM_\Omega(Y):=\dirlim \frakM_{\Omega,m}=\bigcup_{m\geq 0}\frakM_{\Omega,m}.$$
When $\Omega$ is a singleton, the subscript $\Omega$ will be suppressed.
Elements in $\frakM_\Omega(Y)$ are called {\bf bracketed monomials} in $Y$. With the operators
\begin{equation*}
\lc\ \rc_\omega\colon \frakM_\Omega(Y)\to \frakM_\Omega(Y), u\mapsto \lc u\rc_\omega, \ \omega\in \Omega,
\mlabel{eq:mapp}
\end{equation*}
the pair $(\frakM_\Omega(Y), \{\lc\ \rc_\omega\}_{\omega \in \Omega})$ is an operated monoid. Therefore it linear span $(\bfk\frakM_\Omega(Y), \lc\ \rc_{\omega \in \Omega})$ is an operated $\bfk$-algebra.

\begin{prop}$($\mcite{Gop}$)$
Let $j_Y\colon Y \hookrightarrow \frakM_\Omega(Y)$ denote the natural embedding.
Then the triple $(\bfk\frakM_\Omega(Y),\{\lc\ \rc_\omega\}_\omega, j_Y)$ is the free operated $\bfk$-algebra on $Y$. More precisely, for any operated $\bfk$-algebra $R$ and any set map $f \colon Y\to R$, there is a unique extension of $f$ to a homomorphism $\free{f}\colon \bfk\frakM_\Omega(Y)\to R$ of operated $\bfk$-algebras.
\mlabel{pp:freetm}
\end{prop}

\subsection {The construction of free Rota-Baxter algebras.}

Consider $\frakM_\Omega(Y)$ with $\Omega=\{\omega\}$ being a singleton. Denote $P(u):=\lc u\rc:=\lc u\rc_\omega, u\in \frakM(Y)$.
For a nonempty set $Y$ and nonempty subsets $U$ and $V$ of $\frakM(Y)$, define the {\bf alternating products of $U$ and $V$} to be the following subsets of $\frakM(Y)$
\begin{align}
\Lambda(U,V):= \left(\bigcup_{r\geq 0} (UP(V))^rU\right) \bigcup  \left(\bigcup_{r\geq 1} \big(UP(V)\big)^r\right) \bigcup  \left(\bigcup_{r\geq 0} (P(V)U)^rP(V) \right) \bigcup \left(\bigcup_{r\geq 1} (P(V)U)^r\right ).
\mlabel{Eq:fourparts}
\end{align}
With these notations, define
$\Lambda_0(Y) = M(Y)$ to be the free monoid on $Y$ and, for $m\geq 1$, define
$$\Lambda_m(Y) = \Lambda(S(Y), \Lambda_{m-1}(Y)) \cup \{1\}.$$
Then $\Lambda_m(Y), m\geq 0,$ define an increasing sequence
and we define the set of {\bf Rota-Baxter words} to be
$$\ncrbw(Y) := \Lambda_{\infty}(Y):= \cup_{m\geq 0} \Lambda_m(Y).$$
Each $1\neq u\in \ncrbw(Y)$ can be uniquely expressed as $u = u_1\cdots u_m$,
where $u_1,\cdots,u_m$ are alternately in $S(Y)$ and $P(\ncrbw(Y))$.
The {\bf depth} $\dep(u)$ of $u$ is defined to be the least $m\geq 0$ such that $u$ is contained in $\Lambda_m(Y)$.
Define
\begin{equation*}
P_Y\colon  \ncrbw(Y)\to \ncrbw(Y), \quad u\mapsto \lc u\rc, \quad u\in \ncrbw(Y).
\mlabel{eq:freerbp}
\end{equation*}

Let $I_{\mathrm{RB}}(Y)$ denote the operated ideal of $\bfk\frakM(Y)$ generated by elements of the form
$$ \lc u\rc \lc v\rc- \lc u\lc v\rc\rc-\lc \lc u\rc v\rc - \lambda \lc uv\rc, \quad u, v\in \bfk\frakM(Y).$$
By~\mcite{EG,Gub} where $\bfk\rbw(Y)$ is denoted by $\ncsha(Y)$, the composition
\begin{equation}
\bfk\rbw(Y) \to \bfk\frakM(Y) \to \bfk\frakM(Y)/I_{\mathrm{RB}}(Y)
\mlabel{eq:frbij}
\end{equation}
is a bijection. Hence (the coset representatives of) the words
  in~$\rbw(Y)$ form a linear basis of the free Rota-Baxter algebra on $Y$.
Further, write
\begin{equation}
\red:= \alpha \circ \eta\colon  \bfk\frakM(Y)\to \bfk\frakM(Y)/I_{\mathrm{RB}}(Y)
\to \bfk\ncrbw(Y),
\mlabel{eq:red}
\end{equation}
where $\eta:\bfk\frakM(Y)\to \bfk\frakM(Y)/I_{\mathrm{RB}}$ is the quotient map and $\alpha:\bfk\frakM(Y)/I_{\mathrm{RB}}\to \bfk\ncrbw(Y)$ is the inverse of the linear bijection in Eq.~(\mref{eq:frbij}).

Define a product $\diamondsuit$ on $\bfk\ncrbw(Y)$ as follows. Let
$u=u_1u_2\cdots u_s$ and $v= v_1v_2\cdots v_t$ be two Rota-Baxter
words, where $u_i$ for $1\leq i\leq s$ and $v_j$ for $1\leq j\leq t$
are alternately in $S(Y)$ and $\lc \rbw(Y) \rc$.
\begin{enumerate}
\item
If ~$s=t=1$ and hence $u, v\in S(Y)\cup\lc \rbw(Y)\rc$,  then define
\begin{equation}
u\diamondsuit v:=\left\{\begin{array}{ll}
uv, & u \text{ or } v \in S(Y), \\
\red(\lc \tilde{u}\rc\lc \tilde{v}\rc)= \red(\lc B(\tilde{u},\tilde{v})\rc)=\lc \red(B(\tilde{u},\tilde{v}))\rc, & u=\lc \tilde{u}\rc, v=\lc \tilde{v}\rc \in \lc \rbw(Y)\rc, \end{array}\right.
 \mlabel{eq:diapr}
\end{equation}
where $B(\tilde{u}, \tilde{v})= \tilde{u}\lc \tilde{v}\rc + \lc \tilde{u} \rc \tilde{v} + \lambda \tilde{u}\,\tilde{v}$.
\item
If  $s>1$ or $t>1$, then define $$u\diamondsuit v:= u_1u_2\cdots (u_s \diamondsuit v_1)v_2\cdots v_t,$$
where $u_s \diamondsuit  v_1$  is defined by Eq.~(\mref{eq:diapr}) and the remaining products are given by concatenation together with $\bfk$-linearity when $u_s\diamondsuit v_1$ is a linear combination.
\end{enumerate}

We call $\ncrbw(\dx)$ the set of {\bf differential Rota-Baxter (DRB) monomials} on $X$.

\begin{theorem}
\begin{enumerate}
\item\mlabel{it:freediffrba} $($\cite{EG}$)$ Let $Y$ be a set. Then
  $(\bfk\ncrbw(Y),\diamondsuit,P_Y)$ is the free Rota-Baxter algebra
  on $Y$.
\item
$($\cite{GK3}$)$
Let $X$ be a set and $(\bfk \langle  \Delta X\rangle,d)$ the differential algebra of weight $\lambda$ on $X$ in Theorem~\mref{thm:diff}.\mref{it:commfreediff}. There is a unique extension $d_{\Delta X}$ of $d$ to $\bfk \ncrbw(\Delta X)$ such that $(\bfk \ncrbw (\Delta X), d_{\Delta X}, P_{\Delta X})$, together with $j_X\colon \bfk \langle  \dx\rangle  \hookrightarrow \bfk\ncrbw(\Delta X)$, is the free differential Rota-Baxter
    $\bfk$-algebra of weight $\lambda$ on the differential algebra $\bfk \langle  \dx \rangle $.
\label{it:bk3}
\end{enumerate}
\mlabel{thm:freediffrb}
\end{theorem}

In the same fashion, one obtains~$\ncrbw(\Delta_n X))$, called the set of {\bf DRB monomials of order $n$} on $X$, as a basis of $\bfk
  \ncrbw(\Delta_n X)$ by applying~\mref{it:freediffrba}
  to~$Y:=\Delta_n X, n\geq 1$.
 We note that in $\bfk \ncrbw(\Delta_n
X)$, the property $d^{n+1}(u)=0$ only applies to $u\in X$. For
example, taking $n=1$, then $d^2(x)=0$. But $d(\lc x\rc)= x$ and hence
$d^2(\lc x \rc)=d(x)=x^{(1)}\neq 0$.

\subsection{Free integro-differential algebras}
\mlabel{ss:red}

From the universal property of $\bfk\frakM(Y)$, we obtain the following result on free integro-differential algebra, by general principles of universal algebra~\mcite{BN,Co}.

\begin{prop}
Let $X$ be a set. Let $\Omega=\{d, P\}$ and denote $d(u):=\lc u\rc_d, P(u):=\lc u\rc_P$\,.
Let $J_{\ID}=J_{\ID,X}$ be the operated ideal of $\bfk\frakM_\Omega(X)$ generated by the set
$$\left\{\left . \begin{array}{l}
d(uv)-d(u)v-ud(v)-\lambda d(u)d(v),\\
d(1),\\
(d\circ P)(u)-u, \\
P(d(u)P(v))- uP(v)+ P(uv)+ \lambda P(d(u)v),\\
P(P(u)d(v))- P(u)v + P(uv) + \lambda P(ud(v))
\end{array} \,\right|\, u, v\in \frak\frakM_{\Omega}(X)\right\}.$$
Then the quotient operated algebra
$\bfk\frakM_\Omega(X)/J_{\ID}$, with the quotient of the operator $d$
and $P$, is the free integro-differential algebra on $X$.
\mlabel{pp:freerb}
\end{prop}

Our main purpose in this paper is to give an explicit construction of the free integro-differential algebra by determining a canonical subset of $\frakM_\Omega(X)$. The construction is given in Theorem~\mref{thm:gsb}.

We will achieve this construction in several steps. First let $J_{\mathrm{DRB}} = J_{\mathrm{DRB},X}$ denote the operated ideal of $\bfk\frakM_\Omega(X)$ generated by the set
$$\left\{\left . \begin{array}{l}
d(uv)-d(u)v-ud(v)-\lambda d(u)d(v),\\
d(1),\\
(d\circ P)(u)-u, \\
P(u)P(v)-P(uP(v))-P(P(u)v)-\lambda P(uv) \end{array} \,\right|\, u, v\in \frak\frakM_{\Omega}(X)\right\}.$$
Then the quotient operated algebra $\bfk\frakM_{\Omega}(X)/J_{\mathrm{DRB}}$, with the quotient operators $d$ and $P$, is the free differential Rota-Baxter algebra on $X$. Its explicit construction is given in~\mcite{GK3} and recalled in Theorem~\mref{thm:freediffrb}:
$$\bfk\frakM_{\Omega}(X)/J_{\mathrm{DRB}} \cong \bfk \ncrbw(\Delta X),$$
as the free Rota-Baxter algebra on the free differential algebra
$\bfk\langle \Delta X \rangle$ on $X$.

By a simple substitution of $u$ by $P(u)$ in the integro-differential identity in Eq.~(\mref{eq:ibpl}), we see that an integro-differential algebra is a differential Rota-Baxter algebra~\mcite{GRR}. Thus $J_{\ID}$ contains $J_{\mathrm{DRB}}$. Let $I_{\ID}$ denote the image of $J_{\ID}$ under the quotient map $\bfk\frakM_{\Omega}(X)\to \bfk \ncrbw(\Delta X)$, then we have
$$ \bfk \frakM_{\Omega}(X)/J_{\ID}\cong \bfk\ncrbw(\Delta X)/I_{\ID}.$$
Further,  $I_{\ID}$ is the differential Rota-Baxter ideal of $\ncrbw(\Delta X)$ generated by the set
$$\left\{\left . \begin{array}{l}
P(d(u)P(v))- uP(v)+ P(uv)+ \lambda P(d(u)v),\\
P(P(u)d(v))- P(u)v + P(uv) + \lambda P(ud(v))
\end{array} \,\right|\, u, v\in \ncrbw(\Delta X)\right\}.$$

Thus to obtain an explicit construction of the free integro-differential algebra $\bfk \frakM_{\Omega}(X)/J_{\ID}$ by providing a canonical subset of $\frakM_{\Omega}(X)$ as
a basis (of coset representatives) of the quotient, we just need to determine a canonical subset of $\ncrbw(\Delta X)$ as a basis of the quotient $\bfk\ncrbw(\Delta X)/I_{\ID}$.

However, in order to apply the Gr\"obner-Shirshov basis method, we need a monomial (well) order on $\ncrbw(\Delta X)$ which is easily seen to be nonexistent: Suppose $x>P(x)$,
then we have $x>P(x)>\cdots >P^n(x)>\cdots$ leading to an infinite descending chain. Suppose $P(x)>x$, then we have $x>d(x)$, again leading to an infinite descending chain $x>d(x)\cdots >x^{(n)}>\cdots$.
To overcome this difficulty, we consider, for each $n\geq 1$, the free Rota-Baxter algebra $\bfk \ncrbw(\Delta_n X)$ on the truncated differential algebra $\bfk[\Delta_n X]$ in Theorem~\mref{thm:diff}.\mref{it:diffordn} and construct
an explicit basis of the quotient $\bfk\ncrbw(\Delta_nX)/I_{\ID,n}$
where $I_{\ID,n}$ is the differential Rota-Baxter ideal of the Rota-Baxter
algebra $\bfk\ncrbw(\Delta_nX)$ generated by the set
\begin{equation}
\left\{\left . \begin{array}{l}
\phi_1(u,v):=P(d(u)P(v))- uP(v)+ P(uv)+ \lambda P(d(u)v),\\
\phi_2(u,v):=P(P(u)d(v))- P(u)v + P(uv) + \lambda P(ud(v))
\end{array} \,\right|\, u, v\in \ncrbw(\Delta_n X)\right\}.
\mlabel{eq:ibpn}
\end{equation}

Then as $n$ goes to infinity, the above explicit basis will give the desired basis of $\bfk\ncrbw(\Delta X)/I_{\ID}$ and hence of $\bfk \frakM_{\Omega}(X)/J_{\ID}$. See the proof of Theorem~\mref{thm:gsb} for details of this last step.

\section{Weakly monomial order}
\mlabel{sec:monorder}
Write $\nr:= \calr(\dnx)$.

\begin{defn}
{\rm Let $X$ be a set, $\star$ a symbol not in $X$ and $\Delta_n
X^\star := \Delta_n (X\cup \{\star\})$.
\begin{enumerate}
\item
A {\bf
$\star$-DRB monomial on $\Delta_n X$} is defined to be an expression in $\ncrbw(\Delta_n X^\star)$ with exactly one
occurrence of $\star$. We let $\nrs$ denote the set of all $\star$-DRB monomials on $\Delta_n X$.
\item
For $q\in \nrs$ and
$u\in \nr$, we define
$$q|_u := q|_{\star \mapsto u}$$
to be the bracketed monomial in $\frakM (\Delta_n X)$ obtained by replacing the letter $\star$ in $q$ by
$u$. We call $q|_u$ a {\bf $u$-monomial on $\Delta_n X$}.
\item
For $s=\sum_i c_i u_i \in \bfk
\nr$ with $c_i\in \bfk$, $u_i\in \nr$
and $q\in \nrs$, define
$$q|_s := \sum_i c_i q|_{u_i},$$
which is in $\bfk\frakM(\Delta_n X)$. We call $q|_s$ an {\bf $s$-monomial} on $\Delta_n X$.
This applies in particular when $s$ is a monomial.
\end{enumerate}
}
\mlabel{def:drbm}
\end{defn}
We note that the $u$-monomial $q|_u$ from a $\star$-DRB monomial $q$ might not be a DRB monomial.
For example, $q=P(x)\star$ is in $\nrs$ and $u=P(x)$ is in $\nr$ where $x\in X$. But the $u$-monomial $q|_u=P(x)P(x)$ is not in $\nr$.

By the same argument as in the commutative case~\mcite{GGZ}, we have 
\begin{lemma}
Let $S$ be a subset of $\bfk \nr $ and $\mathrm{Id(S)}$
be the differential Rota-Baxter ideal of $\bfk \nr$ generated by $S$.
We have
$$\mathrm{Id(S)} = \left\{\left. \sum_i c_i q_i | _{s_i} \, \right | c_i\in \bfk,  q_i\in \nrs, s_i\in S \right\}.$$
\mlabel{lemma:operator ideal}
\end{lemma}

We now refine the concept of $\star$-DRB monomials. 
\begin{defn}
If $q = p|_{d^\ell(\star)}$ for some $p\in \ncrbw^\star(\Delta_n
X)$ and $\ell\in \mathbb{Z}_{\geq 1}$, then we call $q$ a {\bf type I $\star$-DRB monomial}. Let $\calr_{n,\ii}^\star$ denote the set of type I $\star$-DRB monomials on
$\Delta_n X$ and call
$$\calr_{n,\iii}^\star := \nrs \setminus \calr_{n,\ii}^\star$$
the set of {\bf type II $\star$-DRB monomials}.
\end{defn}

\begin{defn}
Let $<$ be a linear order on $\ncrbw(\Delta_nX)$, $q\in \nrs$ and $s\in
\mathbf{k}\nr$.
\begin{enumerate}
\item
For any $0\neq f\in \bfk \nr$, let $\lbar{f}$ denote the leading term of $f$:
$f = c \overline{f} + \sum_{i} c_iu_i$, where $0\neq  c, c_i\in
\bfk$, $u_i\in \nr$, $u_i< \overline{f}$. Furthermore, $f$ is called {\bf monic} if $c=1$.
\item
Write
$$\overline{q |_s} := \overline{\red(q|_{s}}),$$
where $\red\colon \bfk\frakM(\Delta_nX)\to\bfk \nr$ is the reduction map in Eq.~(\mref{eq:red}).
\item
The element $q|_s\in \bfk\nr$ is called {\bf normal} if
$q|_{\lbar{s}}$ is in $\nr$. In other words, if $\red(q|_{\lbar{s}}) = q|_{\lbar{s}}$. \mlabel{it:normal}
\end{enumerate}
\mlabel{def:normaldef}
\end{defn}

\begin{remark} \begin{enumerate}
\item By definition, $q|_s$ is normal if and only if
$q|_{\lbar{s}}$ is normal if and only if the $\lbar{s}$-DRB monomial $q|_{\lbar{s}}$ is already a DRB monomial, that is, no further reduction in $\bfk\nr$ is possible.
\item
Examples of not normal (abnormal) $s$-DRB monomials are
\begin{enumerate}
\item $q=\star P(x)$ and $\bar{s}=P(x)$, giving $q|_{s}=P(x)P(x)$, which is reduced to $P(xP(y))+P(P(x)y)+\lambda P(xy)$ in $\bfk\nr$;
\item $q=d(\star)$ and $\bar{s}=P(x)$, giving $q|_{\bar{s}}=d(P(x))$, which is reduced to $x$ in $\bfk\nr$;
\item $q=d(\star)$ and $\bar{s}=x^2$, giving $q|_{\bar{s}}=d(x^2)$, which is reduced to $2xx^{(1)}+\lambda (x^{(1)})^2$ in $\bfk\nr$;
\item $q=d^n(\star)$ and $\bar{s}=d(x)$, giving $q|_{\bar{s}}=d^{n+1}(s)$, which is reduced to $0$ in $\bfk\nr$.
\end{enumerate}
\end{enumerate}
\mlabel{rk:ex}
\end{remark}

\begin{defn}
A {\bf weakly monomial order} on $\nr$ is a well order
$<$ satisfying
$$
u < v\, \Rightarrow
\overline{q|_u} < \overline{q|_v} \text{ if either } q \in \calr_{n,\iii}^\star, \text{ or } q \in
\ncrbw^{\star}_{n,\ii} \text{ and } q|_v  \text{ is normal}
$$
for $u, v\in \nr$.
\mlabel{defweakmonomial}
\end{defn}

Let $X$ be a well-ordered set. Let $n\geq 0$ be given. First, we extend the order on $X$ to $\dx$ and $\dnx$. For $x_0^{(i_0)}, x_1^{(i_1)}\in \Delta X$ (resp. $\Delta_n X$) with $x_0, x_1\in X$, define
\begin{equation}
x_0^{(i_0)} < x_1^{(i_1)} \left(\text{resp. } x_0^{(i_0)}<_n x_1^{(i_1)}\right) \Leftrightarrow (x_0,-i_0) < (x_1, -i_1) \quad \text{
lexicographically}.
\mlabel{eq:difford}
\end{equation}
For example $x^{(2)} < x^{(1)}< x$. Also, $x_1<x_2$ implies
$x_1^{(2)} < x_2^{(2)}$. Then by~\mcite{BN}, the order $<_n$ is a well order on $\Delta_n X$. Next, we extend the well order on $\Delta_n X$ to a weakly monomial order on $\nr$.

We adapt the order defined in \mcite{BCD} to the case when the set is
taken to be $\dnx$ and when the order is restricted to
$\nr$. For any $u\in \nr$ and for a
set $T\subseteq \dnx \cup \{P\}$, denote by $\dg_{T}(u)$ the number of
occurrences of $t\in T$ in $u$. Let
$$\dg(u) = (\dg_{P \cup \dnx} (u), \dgp(u) ).$$
We order $\dg(u)$ lexicographically. If $u\in \dnx \cup P(\nr)$, then $u$ is called {\bf indecomposable}. For any $u\in \nr$, $u$ has a {\bf standard form:}
\begin{equation}
u = u_0\cdots u_k, \text{ where } u_0, \cdots, u_k \text{ are indecomposable.}
\mlabel{eq:sdec}
\end{equation}
Now we set up an order $<_n$ on $\nr$ as follows. Let $u,v\in \nr$. If $\dg(u) <\dg(v)$, then $u<_nv$. If $\dg(u)=\dg(v) = (m_1,m_2)$, then we define $u<_n v$ by induction on $(m_1,m_2)$ which is at least $(1,0)$. If $(m_1,m_2) = (1,0)$, that is, $u,v \in \dnx$, we use the order in Eq~(\mref{eq:difford}). Let $(m_1,m_2) > (1,0)$ be given, and assume the order is defined for all $(m_1^\prime, m_2^\prime) < (m_1,m_2)$ and consider $u, v$ with $\dg(u)=\dg(v)=(m_1,m_2)$.
If $u,v\in P(\nr)$, say $u = P(\tilde{u})$ and $v = P(\tilde{v})$, then define $u<_n v$ if and only if $\tilde{u}<_n \tilde{v}$ where the latter is defined by the induction hypothesis. Otherwise, let $u=u_0\cdots u_k$ and $v = v_0\cdots v_\ell$ be the standard forms with $k>0$ or $\ell >0$. Then define $u<_nv$ if and only if $(u_0,\cdots, u_k) < (v_0,\cdots, v_\ell)$ lexicographically. Here the latter is again defined by the induction hypothesis.

We next show that the order $<_n$ defined above is a weakly monomial
order on $\nr$. Recall the following lemma from~\mcite{BCD} on
$\mathcal{R}(X)$ which still applies when it is restricted to $\nr$.

\begin{lemma} (\mcite{BCD} Lemma 3.3)
If $u<_nv$ with $u,v \in \nr$, then $\lbar{uw} <_n \lbar{vw}$ and $\lbar{wu} <_n \lbar{wv}$ for any $w\in \nr$.  \mlabel{lemma:mulcom}
\end{lemma}

\begin{lemma}
Let $\ell\geq 1$ and $s\in
\nr$. Then $d^\ell(\star)|_{s}$ is normal if and only if $s\in
\Delta_{n-\ell} X$. \mlabel{lemma:diffnormal}
\end{lemma}

\begin{proof}
If $s \in \Delta_{n-\ell} X$, then $d^{\ell}(s)$ is in $\Delta_nX$ and hence
$d^\ell(\star)|_{s}$ is normal. Conversely, if $s \notin
\Delta_{n-\ell} X$, then either $s \notin \dnx$ or $s \in \Delta_n X \setminus \Delta_{n-\ell} X$. In both cases we have that $d^\ell(\star)|_{s}$ is not normal. See Remark~\mref{rk:ex}.
\end{proof}

\begin{lemma}
\label{leading diffRB}  Let $u, v\in \nr$ and $\ell\in \mathbb{Z}_{\geq 1}$. If $u <_n v$ and $d^\ell(\star)|_v$ is normal,
then $\lbar{d^\ell(u)} <_n \lbar{d^\ell(v)}$. \mlabel{lemma: leading
diffRB}
\end{lemma}

\begin{proof}
We prove the result by induction on $\ell$. We first consider $\ell
=1$ and prove $\lbar{d(u)} <_n \lbar{d(v)}$. Since $d(\star)|_v$ is normal, we have $v=x_1^{(i_1)}\in \Delta_{n-1} X$ by Lemma
\ref{lemma:diffnormal}. Since $u <_n v$, by the definition of $<_n$, we have $u=x_2^{(i_2)}\in \Delta_n X$ with either $x_2 < x_1$ or $x_1 = x_2$ and $i_2> i_1$. Hence $\lbar{d(u)} <_n \lbar{d(v)}$.

Next, suppose the result holds for $1\leq m <\ell$. Then by the induction hypothesis, we have
$$\lbar{d^{\ell}(u)} = \lbar{d(d^{\ell-1}(u))} = \lbar{d(\lbar{d^{\ell-1}(u)})} <_n \lbar{d(\lbar{d^{\ell-1}(v)})}=
\lbar{d(d^{\ell-1}(v))} = \lbar{d^{\ell}(v)}.$$
\end{proof}

\begin{prop}
 The order $<_n$ is a weakly monomial order on $\nr$. \mlabel{lemma:weakmonomial}
\end{prop}

\begin{proof}
Let $u,v\in \nr$ with $u <_n v$ and $q\in
\nrs$. Depending on the location of the symbol $\star$, we have the
following three cases to consider.
\smallskip

\noindent
{\bf Case 1.} Suppose the symbol $\star$ in $q$ is not contained in $P$ or $d$. Then $q = s\star t$ where $s,t\in \nr$. This case is covered by Lemma \mref{lemma:mulcom}

\noindent
{\bf Case 2.} Suppose the symbol $\star$ is contained in $P$. Then $q = sP(p)t$ for some $s,t\in \nr$ and
$p \in \nrs$. This case can be verified by
induction on $\mathrm{dep}(q)$ and the fact that, for $u,v \in \nr$, $u <_n v$ implies $P(u) <_n P(v)$ by the definition of $<_n$.
\smallskip

\noindent
{\bf Case 3.}  The symbol $\star$ is contained in $d$, that is, $q\in \calr_{n,\ii}^\star$. Then $q =
p|_{d^{\ell}(\star)}$ for some $p\in \nrs$ and
$\ell\in \mathbb{Z}_{\geq 1}$. Take such $\ell$
maximal so that $p\in \calr_{n,\iii}^\star$. We need to show that if $u<_nv$ and $q|_v$ is
normal, then $\lbar{q|_{u}} <_n \lbar{q|_{v}}$. But if $q|_v$ is normal then $d^{\ell}(\star)|_v$ is normal. Then by Lemma
\ref{leading diffRB}, we have $\lbar{d^\ell(u)} <_n
\lbar{d^\ell(v)}$. Then by Cases 1 and 2, we have $\lbar{q|_{u}} =
\lbar{p|_{\lbar{d^\ell(u)}}}
<_n\lbar{p|_{\lbar{d^\ell(v)}}} = \lbar{q|_{v}}$. This completes the proof.
\end{proof}

We shall use the weakly monomial order $<_n$ on $\nr$ throughout the
rest of this paper. The following consequence of Proposition~\mref{lemma:weakmonomial} will be applied in
Section~\mref{sec:cd}.

\begin{lemma}
Let $q\in \nrs$ and let $s\in
\mathbf{k}\nr$ be monic. If $q|_s$ is normal, then
$\lbar{q|_s} = q|_{\lbar{s}}$. \mlabel{normalequiv}
\end{lemma}

\begin{proof}
Let $s = \lbar{s} + \sum_i c_i s_i$ where $0\neq c_i\in \bfk$ and $s_i <_n \lbar{s}$.
Then we have $q|_s = q|_{\lbar{s}} + \sum_i c_i q|_{s_i}$. Since $q|_s$ is
normal, it follows that $q|_{\lbar{s}} \in \nr$. Thus
$\lbar{q|_{\lbar{s}}} = q|_{\lbar{s}}$. We consider the following two
cases.

\noindent
{\bf Case 1.} Suppose $q\in\calr_{n,\iii}^\star$. Then
$\lbar{q|_{s_i}} <_n \lbar{q|_{\lbar{s}}} = q|_{\lbar{s}}$ by
Definition \ref{defweakmonomial} and Proposition \mref{lemma:weakmonomial}. This gives
$\lbar{q|_s} = \lbar{q|_{\lbar{s}}}=q|_{\lbar{s}}$.

\noindent
{\bf Case 2.} Suppose $q\in \calr_{n,\ii}^\star$. Since $q|_s$ is mormal, we have $q|_{\lbar{s}}$ is normal and so $\lbar{q|_{s_i}} < \lbar{ q|_{\lbar{s}} } = q|_{\lbar{s}}$
by Definition \mref{defweakmonomial} and Proposition \mref{lemma:weakmonomial}.
Hence $\lbar{q|_s} = q|_{\lbar{s}}$.
\end{proof}

\section{Composition-Diamond lemma}
\mlabel{sec:cd}
In this section, we establish the Composition-Diamond lemma for the free differential Rota-Baxter algebra of order $n$ defined in Theorem~\mref{thm:diff}.

\begin{defn}
Let $X$ be a set, $\star_1$, $\star_2$ two distinct symbols not in
$X$ and $\Delta_n X^{\star_1, \star_2} := \Delta_n (X\cup
\{\star_1,\star_2\})$.

\begin{enumerate}
\item
We define $\ncrbw(\Delta_n X^{\star_1,\star_2})$ in the same way as for $\calr(\dnx)$ with $X$ replaced by $X\cup \{\star_1, \star_2\}$.
\item
 We define a {\bf
$(\star_1,\star_2)$-DRB monomial on
$\Delta_n X$ } to be an expression in $\ncrbw(\Delta_n
X^{\star_1,\star_2})$ with exactly one occurrence of $\star_1$ and
exactly one occurrence of $\star_2$. The set of all $(\star_1,
\star_2)$-DRB monomials on $\Delta_n X$ is denoted by $\ncrbw_n^{\star_1, \star_2}$.
\item
For $q\in\ncrbw_n^{\star_1, \star_2}$ and $u_1, u_2\in \bfk
\nr$, we define
$$q|_{u_1,u_2} := q|_{\star_1 \mapsto u_1, \star_2 \mapsto u_2}$$
to be the bracketed monomial obtained by
replacing the letter $\star_1$ (resp. $\star_2$) in $q$ by $u_1$
(resp. $u_2$) and call it a {\bf $(u_1,u_2)$-monomial on $\Delta_n X$ }.
\item
The element $q|_{u_1,u_2}$ is called {\bf normal} if $q|_{\lbar{u}_1,\lbar{u}_2}$ is in $\nr$. In other words, if $\red(q|_{\lbar{u}_1,\lbar{u}_2}) = q|_{\lbar{u}_1,\lbar{u}_2}$.
\end{enumerate}
\mlabel{def:dDRBmon}
\end{defn}

A $(u_1,u_2)$-DRB monomial on $\Delta_n
X$ can also be recursively defined by
$q|_{u_1,u_2} := (q^{\star_1}|_{u_1})|_{u_2},$
where $q^{\star_1}$ is $q$ when $q$ is regarded as a
$\star_1$-DRB monomial on the set
$\Delta_n X^{\star_2}$. Then $q^{\star_1}|_{u_1}$ is in
$\ncrbw^{\star_2}(\Delta_n X)$. Similarly, we have
$q|_{u_1,u_2} := (q^{\star_2}|_{u_2})|_{u_1}.$

\begin{defn}
\begin{enumerate}
\item Let $u,w \in \nr$. We call $u$ a {\bf subword} of $w$ if there is a $q\in \nrs$ such that $w=q|_u$.
\item
Let $u_1$ and $u_2$ be two subwords of $w$. Then $u_1$ and $u_2$ are called {\bf separated} if $u_1, u_2\in \nr$ and there is a $q\in \ncrbw^{\star_1,
\star_2}(\Delta_n X)$ such that $w=q|_{u_1,u_2}$.
\item
Let $u = u_1\cdots u_k\in \nr $ be the standard form. The integer $k$ is called
the {\bf breadth} of $u$ and is denoted by $\mathrm{bre}(u)$.
\item
Let $f,g\in \nr$. A pair $(u,v)$ with $u,v\in \nr$ is called an {\bf intersection pair} for $(f,g)$ if $w:=fu=vg$ or $w:=uf=gv$ is a differential Rota-Baxter monomial and satisfies
$\max\{\mathrm{bre}(f),\mathrm{bre}(g)\}<\mathrm{bre}(w)<\mathrm{bre}(f) + \mathrm{bre}(g)$. In this case $f$ and $g$ are called {\bf overlapping}.
\end{enumerate}
\end{defn}

There are three kinds of compositions.

\begin{defn}
Let $f,g\in \mathbf{k}\nr$ be monic with respect to $<_n$.

\begin{enumerate}
\item If $\overline{f}\in \nr P(\nr)$, then define a {\bf composition of
right multiplication} to be $fu$ where $u\in P(\nr)\nr$. We similarly define a {\bf composition of
left multiplication}.
\item If there is an intersection pair $(u,v)$ for $(\overline{f},
\overline{g})$ with $w:=\overline{f}u=v\overline{g}$ (resp. $w:=u\overline{f}=\overline{g}v$), then we denote
$$(f,g)_w :=(f,g)^{u,v}_{w} := fu-vg\ (\text{resp. }uf-gv)$$ and call it an {\bf intersection
composition} of $f$ and $g$.
\item If there is $q\in \nrs$ such that
$w:=\overline{f}=q|_{\overline{g}}$, then we denote $(f,g)_w
:=(f,g)^{q}_{w} := f-q|_g$ and call it an {\bf inclusion
composition} of $f$ and $g$ with respect to $q$. Note that in this case, $q|_g$ is normal.
\end{enumerate}
\end{defn}

In the last two cases, $w$ is called the {\bf ambiguity} of the composition.

\begin{defn}
Let $S\subseteq \mathbf{k}\nr$ be a set of monic differential Rota-Baxter polynomials and $w\in \nr$.

\begin{enumerate}
\item  An element $g$ in $\bfk \nr$ is called {\bf trivial modulo $[S]$} if
$g = \sum_i c_i q_i|_{s_i},$
where, for each $i$, we have $0\neq c_i\in \bfk$, $q_i\in \nr^\star$, $s_i\in S$ such that $q_i|_{s_i}$
is normal and $q_i|_{\overline{s_i}} \leq_n \overline{g}$.
If this is the case, we write $g \equiv 0  \text{ mod }  [S].$

\item  The composition of right (resp. left) multiplication $fu$ (resp. $uf$) is called {\bf trivial modulo
$[S]$} if $fu \equiv 0 $ mod $[S]$ (resp. $uf \equiv 0 $ mod $[S]$).

\item For $u,v\in \mathbf{k}\nr$, we call $u$ and $v$
{\bf congruent modulo $[S,w]$} and denote this by $$u \equiv v \text
{ mod } [S,w]$$
if $u-v=0$, or if $u-v = \sum_{i}c_iq_i|_{s_i}$, where $0\neq c_i\in\bfk$, $q_i\in \nrs$, $s_i\in S$ such that $q_i|_{s_i}$ is normal and $q_i|_{\lbar{s_i}} <_n w$.
\item For $f,g\in \mathbf{k}\nr$ and suitable $u,v$ or $q$
that give an intersection composition $(f,g)^{u,v}_{w}$ or an including composition $(f,g)^{q}_{w}$, the composition is called
{\bf trivial modulo $[S,w]$} if $$(f,g)^{u,v}_{w} \text{ or }
(f,g)^{q}_{w}\equiv 0 \text{ mod } [S,w].$$

\item The set $S\subseteq \mathbf{k}\nr$ is a {\bf Gr\"{o}bner-Shirshov basis} if all
compositions of right multiplication and left multiplication are trivial modulo $[S]$, and, for $f, g\in S$, all
intersection compositions $(f,g)^{u,v}_{w}$ and all inclusion compositions $(f,g)^{q}_{w}$ are trivial modulo $[S,w]$.
\end{enumerate}
\mlabel{def:trivial}
\end{defn}

We give some preparatory lemmas before establishing the Composition-Diamond Lemma.

\begin{lemma}
Let $S\subseteq \mathbf{k}\nr$ with $d(S)\subseteq S$. If each composition of left multiplication and right multiplication of
$S$ is trivial modulo $[S]$, then $q|_s$ is trivial modulo $[S]$ for every $q \in \calr_n^\star$ and $s
  \in S$.
\mlabel{lemma:normalexp}
\end{lemma}
\begin{proof}
We have the following two cases to consider.
\smallskip

\noindent
{\bf Case 1.}  $q\in \calr_{n,\iii}^\star$. This case is similar to the proof of Lemma 3.6 in \mcite{BCD}.

\noindent
{\bf Case 2.} $q\in \calr_{n,\ii}^\star$. Then $q =
p|_{d^{\ell}(\star)}$ for some $p\in \nrs$ and
$\ell{\geq 1}$. Choose such an $\ell$ to be maximal so that
$p$ is in $\calr_{n,\iii}^\star$. Since $d(S) \subseteq S$, by Case 1 that has been proved above, the result holds.
\end{proof}

\begin{lemma}
Let $S \subseteq \bfk \nr$ with $d(S) \subseteq S$ be a Gr\"{o}bner-Shirshov basis. Let $s_1,s_2\in S$, $q_1,q_2 \in \nrs$ and $w\in \nr$ such that $w= q_1|_{\overline{s_1}} = q_2|_{\overline{s_2}}$,
where $q_i|_{s_i}$ is normal for $i=1,2$. If $\overline{s_1}$ and
$\overline{s_2}$ are separated in $w$, then $q_1|_{s_1} \equiv q_2|_{s_2}$ mod $[S,w]$.
\mlabel{lemma:separated}
\end{lemma}

\begin{proof}
Let $q\in \ncrbw_n^{\star_1, \star_2}$ be the
$(\star_1,\star_2)$-DRB monomial obtained by replacing the occurrence of $\overline{s_1}$ in $w$ by
$\star_1$ and the occurrence of $\overline{s_2}$ in $w$ by
$\star_2$. Then we have
$$q^{\star_1}|_{\overline{s_1}} = q_2, q^{\star_2}|_{\overline{s_2}}
= q_1 \text{ and } q|_{\overline{s_1}, \overline{s_2}} =
q_1|_{\overline{s_1}} = q_2|_{\overline{s_2}},$$
where in the first two equalities, we have identified $\ncrbw_n^{\star_2}$ and
$\ncrbw_n^{\star_1}$ with $\nrs$. Let
$s_1-\overline{s_1} = \sum_{i} c_iu_i$ and $s_2-\overline{s_2} =
\sum_{j} d_jv_j$ with $0\neq c_i,d_j\in \mathbf{k}$ and $u_i,v_j\in
\nr$ such that $u_i <_n \lbar{s_1}$ and $v_j <_n \lbar{s_2}$. Then by the linearity of $s_1$ and $s_2$ in
$q|_{s_1,s_2}$, we have
\begin{align*}
q_1|_{s_1} - q_2|_{s_2} &= (q^{\star_2}|_{\overline{s_2}})|_{s_1} -
(q^{\star_1}|_{\overline{s_1}})|_{s_2}\\
&= q|_{s_1,\overline{s_2}} - q|_{\overline{s_1},s_2}\\
&=q|_{s_1,\overline{s_2}} - q|_{s_1,s_2} + q|_{s_1,s_2} -
q|_{\overline{s_1},s_2}\\
&= -q|_{s_1, s_2-\overline{s_2}} + q|_{s_1-\overline{s_1},s_2}\\
&= -(q^{\star_2}|_{s_2-\overline{s_2}})|_{s_1} +
(q^{\star_1}|_{s_1-\overline{s_1}})|_{s_2}\\
&= -\sum_{j} d_j(q^{\star_2} | _{v_j})|_{s_1} + \sum_{i}
c_i(q^{\star_1} | _{u_i})|_{s_2}\\
&= -\sum_{j} d_j q|_{s_1,v_j} + \sum_{i} c_i q|_{u_i, s_2}.
\end{align*}
From Lemma \mref{lemma:normalexp}, for each $j$, we may suppose that
$$q|_{s_1,v_j} = (q|_{s_1})|_{v_j} = \sum_{\ell} d_{j\ell} p_\ell|_{v_{j\ell}},$$
where $0\neq d_{j\ell} \in \bfk$, $p_\ell \in \nrs$ , $v_{j\ell} \in S$ such that $p_\ell|_{v_{j\ell}}$ is normal and $\lbar{p_\ell|_{v_{j\ell}}} \leq_n \lbar{(q|_{s_1})|_{v_j}} = \lbar{q|_{s_1,v_j}}$.
Since $(q^{\star_1}|_{s_1})|_{\lbar{s_2}} = q|_{s_1, \lbar{s_2}} = (q^{\star_2}|_{\lbar{s_2}})|_{s_1} = q_1|_{s_1}$ is normal and $v_j <_n\lbar{s_2}$, by Definition~\mref{defweakmonomial} and Proposition~\mref{lemma:weakmonomial}, we have
$$\lbar{q|_{s_1, v_j}} = \lbar{(q^{\star_1}|_{s_1})|_{v_j}} <_n \lbar{(q^{\star_1}|_{s_1})|_{\lbar{s_2}}}
 = \lbar{q_1|_{s_1}} = q_1|_{\lbar{s_1}} = w.$$
 So we have
$$\lbar{p_\ell|_{v_{j\ell}}} \leq_n w.$$
With a similar argument to the case of $q|_{u_i, s_2}$, we can obtain that $q_1|_{s_1} \equiv q_2|_{s_2}$ mod $[S,w]$.
\end{proof}

For $k\geq 1$, write $\frakM_k:=\frakM_{\Omega,k}(\Delta_n X)$ where
$\Omega=\{d,P\}$.  For $q\in \nrs$, we define the {\bf depth
  $\mathrm{dep}_{\star}(q)$ of $\star$ in $q$} by induction on $k\geq 0$ such
that $q\in \nrs \cap \frakM_k$.  Let $k=0$. Then $q\in M(\dnx^\star)$ and we
define $\mathrm{dep}_\star (q)=0$. Suppose $\mathrm{dep}_\star(q)$ has been
defined for $q\in \nrs \cap \frakM_m, m\geq 0,$ and consider $q\in \nrs\cap
\frakM_{m+1}$. Then we have $q=q_1\cdots q_\ell$ with each $q_i$ in $\Delta_nX
\cup \{\star\}$ or $\lc \frakM(\dnx^\star)\rc\cap \frakM_{m+1}, 1\leq i\leq
\ell,$ and with $\star$ appearing in a unique $q_i$.  Suppose the unique $q_i$
is in $\Delta_nX \cup \{\star\}$. Then define $\mathrm{dep}_\star(q)=0$. Suppose
the unique $q_i$ is in $\lc \frakM(\dnx^\star)\rc\cap \frakM_{m+1}$. Then
$q_i=\lc \tilde{q}_i\rc$ with $\tilde{q}_i\in \frakM(\dnx^\star)\cap
\frakM_{m}$. Thus $\tilde{q}_i$ is in $\nrs\cap \frakM_m$ and
$\mathrm{dep}_\star(\tilde{q}_i)$ is defined by the induction hypothesis. We
then define $\mathrm{dep}_\star(q):=\mathrm{dep}_\star(\tilde{q}_i)+1.$ For
example, $\mathrm{dep}_{\star}(q) = 1$ if $q=P(\star)$ and
$\mathrm{dep}_{\star}(q) = 2$ if $q=P(xP(\star))$.

For the purpose of the proof the next lemma, we describe the relative location of two bracketed subwords in the more precise notion of placements (or occurrences~\mcite{BKV}) in a bracketed word. See~\cite{ZG} for details. But note that we focus on words in $\nr$ as a subset of $\frakM(\dnx)$.

\begin{defn}
Let $w, u\in \nr$ and $q\in \nrs$ be such that $w=q|_u$. Then we call the pair $(u,q)$ a {\bf placement} (or {\bf occurrence}) of $u$ in $w$.
\end{defn}

The pair $(u,q)$ corresponds to the pair $(q,u)$ in~\cite[Chapter 2]{BKV} where $q$ is called the prefix.
We note that a placement $(u,q)$ gives an appearance of $u$ as a subword or subterm of $w=q|_u$. A placement is more precise than a subword since a placement emphasizes the location of a subword. For example $u=x$ has two appearances in $w=x\lc x\rc$ which are differentiated by the two placements $(u,q_1)$ and $(u,q_2)$ where $q_1=\star\lc x\rc$ and $x\lc \star\rc$.

\begin{defn}
Let $w, u_1, u_2\in \nr$ and $q_1, q_2\in \nrs$ be such that
\begin{equation}q_1|_{u_1}=w=q_2|_{u_2}.\mlabel{eq:plas}
\end{equation}
The two \plas $(u_1,q_1)$ and $(u_2,q_2)$ are said to be
\begin{enumerate}
\item
{\bf separated} if there exists an element $q$ in $\nr^{\star_1,\star_2}$ and $a,b \in \nr$ such that $q_1|_{\star_1}=q|_{\star_1,\,b}$, $q_2|_{\star_2}=q|_{a,\,\star_2}$ and $w=q|_{a,\,b}$;
\mlabel{item:bsep}
\item
{\bf nested} if there exists an element $q$ in $\nrs$ such that \ws{either} $q_2=q_1|_q$ or $q_1=q_2|_q$;
\mlabel{item:bnes}
\item
 {\bf intersecting} if there exist an element $q$ in $\nrs$ and elements $a, b, c$
 in $\nr\backslash\{1\}$ such that $w=q|_{abc}$ and either
\begin{enumerate}
\item \mlabel{item:left2}
$ q_1=q|_{\star c}, q_2=q|_{a\star}$; or
\item \mlabel{item:right2}
$q_1=q|_{a\star}, q_2=q|_{\star c}$.
\end{enumerate}
\mlabel{item:bint}
\end{enumerate}
\mlabel{defn:bwrel}
\end{defn}

By taking $u=abc$, it is easy to see that $(u_1,q_1)$ and $(u_2,q_2)$ are intersecting (in case
(i)) if and only if there are $v_1,v_2 \in \nr$ such that $w=q|_{u}$, $u:= u_1
v_1 = v_2 u_2$ and
$$ \max \{  \mathrm{bre}(u_1),\mathrm{bre}(u_2) \} < \mathrm{bre}(u)< \mathrm{bre}(u_1)+\mathrm{bre}(u_2).$$
This corresponds to the above definition via the relations~$(u,v_1,v_2)=(abc,c,a)$.

\begin{theorem}
Let $w$ be a bracketed word in $\nr$. For any two \plas $(u_1, q_1)$ and $(u_2,q_2)$ in $w$, exactly one of the following is true:
\begin{enumerate}
\item
$(u_1, q_1)$ and $(u_2,q_2)$ are separated;
\item
$(u_1, q_1)$ and $(u_2,q_2)$ are nested;
\item
$(u_1, q_1)$ and $(u_2,q_2)$ are intersecting.
\end{enumerate}
\mlabel{thm:thrrel}
\end{theorem}
\begin{proof}
Let $\frakM_{\{P\}}(\Delta X)$ denote the set of bracketed words on the set $\Delta X$ with the bracket given by $P$. By Theorem~\mref{thm:freediffrb}.\ref{it:bk3}, for the Rota-Baxter ideal $J_{\mathrm{RB}}$ of $\bfk \frakM_{\{P\}}(\Delta X)$ generated by the set
$$\{P(u)P(v)-P(uP(v))-P(P(u)v)-\lambda P(uv)\,|\, u, v\in \frakM_{\{P\}}(\Delta X)\},$$
we have
$$\bfk \ncrbw(\Delta X)\cong \bfk\frakM_{\{P\}}(\Delta X)/J_{\mathrm{RB}} \cong \bfk\frakM_{\{P,d\}}(X)/J_{\mathrm{DRB}}.$$
By~\cite[Theorem 4.11]{ZG}, the statement of the present theorem holds when $\nr$ is replaced by $\frakM_{\{P\}}(\Delta X)$.
Since $\ncrbw(\Delta X)$ and hence $\nr$ are subsets of $\frakM_{\{P\}}(\Delta X)$, the statement of the theorem remains true for $\ncrbw(\Delta X)$ and $\nr$.
\end{proof}

Now we are ready to prove the next result.

\begin{lemma}
Let $S\subseteq \mathbf{k}\nr$ with $d(S) \subseteq S$. If $S$ is a Gr\"{o}bner-Shirshov basis, then for each pair
$s_1,s_2\in S$ for which there exist $q_1,q_2\in \nrs$ and $w\in \nr$ such that
$w=q_1|_{\overline{s_1}} = q_2|_{\overline{s_2}}$ with
$q_1|_{s_1}$ and $q_2|_{s_2}$ normal, we have $q_1 | _{s_1} \equiv q_2 |
_{s_2}$ mod $[S,w]$.

\mlabel{lemma:basis}
\end{lemma}

\begin{proof}  Let $s_1,s_2\in S$, $q_1, q_2\in \nrs$ and $w\in \nr$ be such that
$w=q_1|_{\overline{s_1}} = q_2|_{\overline{s_2}}$. Let $(\lbar{s_1},q_1)$ and $(\lbar{s_2}, q_2)$ be the corresponding placements of $w$. By Theorem~\mref{thm:thrrel}, according to the relative location of the placements $(q_1,\lbar{s_1})$ and $(q_2,\lbar{s_2})$ in $w$, we have the following three cases to consider.
\smallskip

\noindent
{\bf Case 1.} The placements $(\lbar{s_1},q_1)$ and $(\lbar{s_2},q_2)$ are separated in
$w$. This case is covered by Lemma \ref{lemma:separated}.
\smallskip

\noindent
{\bf Case 2.} The placements $(\overline{s_1},q_1)$ and $(\overline{s_2},q_2)$ are intersecting in $w$. We only need to consider Case (i) of overlapping since the proof of Case (ii) is similar. Then by the remark after Definition~\mref{defn:bwrel}, there are $u,v \in \nr$ such that $w_1:=\overline{s_1}u=v\overline{s_2}$ is a subword in
$w$, where
$$ \max \{  \mathrm{bre}(\overline{s_1}),\mathrm{bre}(\overline{s_2}) \} < \mathrm{bre}(w_1)< \mathrm{bre}(\overline{s_1})+
\mathrm{bre}(\overline{s_2}).$$
Since $S$ is a Gr\"{o}bner-Shirshov
basis, we have
$$s_1u-vs_2 = \sum_j c_j p_j |_{t_j},$$
where $0\neq c_j\in\bfk$, $t_j \in S, p_j \in \calr^\star_n$ such that $p_j |_{t_j}$ is normal and $\overline{p_j |_{t_j}} = p_j
|_{\overline{t_j}} <_n \overline{s_1}u = v\overline{s_2} =
w_1$.

Let $q\in \ncrbw_n^{\star_1, \star_2}$ be obtained from
$q_1$ by replacing $\star$ by $\star_1$, and the $u$ on the right of
$\star$ by $\star_2$. Let $p\in \nrs$ be obtained
from $q$ by replacing $\star_1\star_2$ by $\star$. Then we have
$$q^{\star_2} |_u = q_1, q^{\star_1} |_v = q_2 \text{ and } p|_{\overline{s_1}u} =
q|_{\overline{s_1}, u} = q_1 |_{\overline{s_1}} =w,$$
where in the first two equalities, we have identified $\calr^{\star_2}_n$ and $\calr^{\star_1}_n$ with $\calr^\star_n$. Thus we have
$$q_1 |_{s_1} - q_2 |_{s_2} = (q^{\star_2}|_u)|_{s_1}
-(q^{\star_1}|_v)|_{s_2} = p|_{s_1u-vs_2} = \sum_j c_j p|_{p_j|_{t_j}} = \sum_j c_j \tilde{p_j}|_{t_j},$$
where $\tilde{p_j} := p|_{p_j} \in \nrs$. By Lemma \mref{lemma:normalexp}, for each $j$, we may suppose that
$$\tilde{p_j}|_{t_j} = \sum_{\ell} c_{j\ell} p_{j\ell}|_{t_{j\ell}},$$
where $0\neq c_{j\ell}\in\bfk$, $t_{jl} \in S, p_{jl} \in \calr^\star_n$, $p_{j\ell}|_{t_{j\ell}}$ is normal and $\lbar{p_{j\ell}|_{t_{j\ell}}} \leq_n \lbar{ \tilde{p_j}|_{t_j}}$. So
$$q_1 |_{s_1} - q_2 |_{s_2} =  \sum_j c_j \tilde{p_j}|_{t_j} = \sum_{j, \ell} c_j c_{j\ell} p_{j\ell}|_{t_{j\ell}}.$$
Since
$\overline{p_j |_{t_j}} <_n w_1$ and $p|_{w_1} =w \in
\nr$
is normal, by Definition \mref{defweakmonomial},
we have
$$ \lbar{ \tilde{p_j} |_{t_j} }  = \lbar{p|_{p_j|_{t_j}}} =   \lbar{p|_{ \lbar{ p_j|_{t_j}}}} <_n \lbar{p|_{w_1}} =
p|_{w_1} = w$$
and so
$$\lbar{ p_{j\ell}|_{t_{j\ell}}} \leq_n \lbar{ \tilde{p_j} |_{t_j} }  <_n  w. $$
Hence
$$q_1 |_{s_1} \equiv q_2 | _{s_2} \text{ mod } [S,w].$$
\smallskip

\noindent
{\bf Case 3.} The placements $(\lbar{s_1},q_1)$ and $(\lbar{s_2},q_2)$ are nested. Without loss of generality, we may suppose $q_2=q_1|_q$ for some $q\in
\nrs$. Then
$q_1|_{\lbar{s_1}}=q_2|_{\lbar{s_2}}=(q_1|_q)|_{\lbar{s_2}}$ and hence
$\overline{s_1}=q|_{\overline{s_2}}$. Since $\overline{s_1}=q|_{\overline{s_2}}
\in \nr$, it follows that $q|_{s_2}$ is normal  by
Definition \ref{def:normaldef} and $\overline{q|_{s_2}} =
q|_{\overline{s_2}}$. For the  inclusion composition
$(s_1,s_2)^q_{\overline{s_1}}$, since $S$ is a Gr\"{o}bner-Shirshov
basis, we have
$$(s_1,s_2)^q_{\overline{s_1}} = s_1-q|_{s_2} = \sum_j c_j p_j |_{t_j},$$
where $0\neq c_j\in\bfk$, $p_j\in \nrs$, $t_j\in
S$ and $p_j|_{t_j}$
is normal with $\overline{p_j|_{t_j}} <_n
\overline{s_1}$.  Thus
$$q_2|_{s_2} - q_1|_{s_1}
= q_1|_{q|_{s_2}} - q_1|_{s_1} = -q_1|_{s_1-q|_{s_2}} = -\sum_j c_j q_1|_{p_j|_{t_j}}
= - \sum_j c_j \tilde{p_j}|_{t_j},
$$
where $\tilde{p_j} := q_1|_{p_j} \in \calr_n^\star$. By Lemma \mref{lemma:normalexp}, for each $j$, we may write
$$\tilde{p_j}|_{t_j} = \sum_{\ell} c_{j\ell} p_{j\ell}|_{t_{j\ell}},$$
where $ 0\neq c_{j\ell} \in \bfk$, $p_{j\ell}|_{t_{j\ell}}$ is normal and $\lbar{p_{j\ell}|_{t_{j\ell}}} \leq_n \lbar{ \tilde{p_j}|_{t_j}}$.
So
$$q_2 |_{s_2} - q_1 |_{s_1}  = - \sum_{j, \ell} c_j c_{j\ell} p_{j\ell}|_{t_{j\ell}}.$$
Since
$\overline{p_j |_{t_j}} <_n \lbar{s_1}$ and $q_1|_{\lbar{s_1}} =w \in \calr_n$
is normal, by Definition \mref{defweakmonomial},
we have
$$ \lbar{ \tilde{p_j} |_{t_j} }  = \lbar{q_1|_{p_j|_{t_j}}} =   \lbar{q_1|_{ \lbar{ p_j|_{t_j}}}} <_n \lbar{ q_1 |_{\lbar{ s_1 } } } =
q_1 |_{\lbar{ s_1 } } = w$$
and so $\lbar{ p_{j\ell}|_{t_{j\ell}}} \leq_n \lbar{ \tilde{p_j} |_{t_j} }  <_n  w.$
Hence $q_2|_{s_2} - q_1|_{s_1} \equiv 0 \text{ mod } [S, w].$

This completes the proof of Lemma~\mref{lemma:basis}.
\end{proof}

\begin{lemma}
Let $S \subseteq \mathbf{k}\nr$ with $d(S) \subseteq S$ and $\mathrm{Irr(S)}:= \nr \setminus \{
q|_{\overline{s}} \mid q\in \nrs, s\in S ,
q|_{s}\text{ is normal } \}$. Then any $f\in \mathbf{k}\nr$ has an expression
$$f = \sum_i c_i u_i +  \sum_j d_j q_j|_{s_j}, $$
where for each $i, j$, we have $0\neq c_i,d_j \in \bfk, u_i\in \mathrm{Irr(S)},
\overline{u_i} \leq_n \lbar{f}$, $q_j\in \nrs$, $s_j\in S$ such that $q_j|_{s_j}$ is normal and
$q_j|_{\lbar{s_j}} \leq_n \lbar{f}$.
\mlabel{lemma:expression}
\end{lemma}

\begin{proof}
Suppose the lemma does not hold and let $f$ be a counterexample with $\lbar{f}$ minimal. Write
$f = \sum_i c_i u_i$ where $0\neq c_i\in \bfk$, $u_i\in
\nr$ and $u_1>_n u_2 >_n \cdots$. If $u_1\in
\mathrm{Irr(S)}$, then let $f_1 := f-c_1u_1$. If $u_1\notin
\mathrm{Irr}(S)$, that is, there exists $s_1\in S$ such that $u_1 = q_1|_{\lbar{s_1}}$ and $q_1|_{s_1}$ is normal, then let $f_1 := f -
c_1q_1|_{s_1}$. In both cases $\lbar{f_1} <_n \lbar{f}$. By the
minimality of $f$, we have that $f_1$ has the desired
expression. Then $f$ also has the desired expression. This is a
contradiction.
\end{proof}

Now we are ready to state and prove the Composition-Diamond Lemma.

\begin{theorem}(Composition-Diamond Lemma) Let $S$ be a set of monic DRB polynomials in $\bfk \nr$ with $d(S) \subseteq S$ and $\mathrm{Id}(S)$ the differential Rota-Baxter ideal of
$\mathbf{k}\nr$ generated by $S$. Then the following
conditions are equivalent:
\begin{enumerate}
\item $S$ is a Gr\"{o}bner-Shirshov basis in
$\mathbf{k}\nr$.
\mlabel{it:cda}
\item If $0\neq f\in \mathrm{Id}(S)$, then $\overline{f} =
q|_{\overline{s}}$, where $q\in \nrs$, $s\in S$
and $q|_s$ is normal.
\mlabel{it:cdb}
\item
The set 
$\mathrm{Irr(S)} = \nr \setminus
\{ q|_{\overline{s}} \mid q\in \nrs, s\in S ,
q|_{s}\text{ is normal} \}$ is a $\mathbf{k}$-basis of
$\mathbf{k}\nr/\mathrm{Id}(S)$.
In other words, $\mathbf{k} \mathrm{Irr}(S) \oplus \mathrm{Id}(S) = \mathbf{k}
\nr$.
\mlabel{it:cdc}
\end{enumerate} \mlabel{thm:CD lemma}
\end{theorem}

\begin{proof} \mref{it:cda} $\Rightarrow$ \mref{it:cdb}: Let $0\neq f\in \mathrm{Id}(S)$.
Then by Lemmas \mref{lemma:operator ideal} and \mref{lemma:normalexp} we have
\begin{equation}\label{ideal}
f = \sum_{i=1}^{k} c_iq_i|_{s_i}, \text{ where } 0\neq c_i \in
\mathbf{k}, q_i\in \nrs, s_i\in S, q_i|_{s_i} \text{ is normal}, 1\leq i\leq k.
\end{equation}
Let $w_i = q_i |_{\lbar{s_i}}, 1\leq i\leq k$. Rearrange the elements in non-increasing order:
$$w_1=w_2=\cdots=w_m >_n w_{m+1} \geq_n \cdots \geq_n w_k.$$
If for each $0\neq f\in \mathrm{Id}(S)$, there is a choice of the
above sum such that $m=1$, then $\overline{f} = q_1|_{\lbar{s_1}}$ and we
are done. Thus assume that the implication (a) $\Rightarrow({\rm b})$ does
not hold. Then there is an $0\neq f\in \mathrm{Id}(S)$ such that for
any expression in Eq.~(\ref{ideal}), we have $m\geq 2$. Fix
such an $f$ and choose an expression in Eq.~(\ref{ideal}) such that
$q_1|_{\overline{s_1}}$ is minimal and such that $m\geq 2$ is minimal. In other words, it has the fewest $q_i|_{s_i}$ such that
$q_i|_{\overline{s_i}}=q_1|_{\overline{s_1}}$. Since $m\geq 2$, we
have $q_1|_{\overline{s_1}}=w_1=w_2=q_2|_{\overline{s_2}}$.

Since $S$ is a Gr\"{o}bner-Shirshov basis in $\mathbf{k}
\nr$, by Lemma \mref{lemma:basis}, we have
$
q_2|_{s_2} - q_1|_{s_1} =\sum_j d_j p_j|_{r_j},
$
with $0\neq d_j\in \mathbf{k}$, $r_j\in S$, $p_j\in \nrs$ and $p_j|_{r_j}$ normal such that $p_j|_{\overline{r_j}} <_n w_1$.
Therefore,
$$
f = \sum_{i=1}^{k} c_iq_i|_{s_i} = (c_1+c_2)q_1|_{s_1} + c_3 q_3|_{s_3} + \cdots + c_m q_m|_{s_m} +
\sum_{i=m+1}^{k} c_iq_i|_{s_i} + \sum_j c_2 d_j p_j|_{r_j}.
$$
By the minimality of $m$, we must have $c_1+c_2 = c_3 =\cdots = c_m
= 0$. Then we obtain an expression of $f$ in the form of
Eq.~(\ref{ideal}) for which $q_1|_{\overline{s_1}}$ is even smaller. This gives the desired contradiction.
\smallskip

\noindent
\mref{it:cdb} $\Rightarrow \mref{it:cdc}$: Clearly $0\in \mathbf{k} \mathrm{Irr}(S)
+\mathrm{Id}(S) \subseteq \mathbf{k} \nr$. Suppose the
inclusion is proper. Then $\mathbf{k} \nr \setminus
(\mathbf{k} \mathrm{Irr}(S) + \mathrm{Id}(S))$ can contain only nonzero
elements. Choose $f\in \mathbf{k} \nr \setminus
(\mathbf{k} \mathrm{Irr}(S) + \mathrm{Id}(S)) $ such that
$$
\overline{f} = \mathrm{min} \{\overline{g} \mid g\in \mathbf{k}
\nr \setminus (\mathbf{k} \mathrm{Irr}(S) +
\mathrm{Id}(S)) \}.
$$
We consider two cases. 
\smallskip

\noindent
{\bf Case 1.} Suppose $\overline{f} \in \mathrm{Irr(S)}$. Then $f\neq
\overline{f}$ since $f\notin \mathrm{Irr(S)}$. By
$\overline{f-\overline{f}} <_n \overline{f}$ and the minimality of
$\overline{f}$, we must have
$$ f-\overline{f} \in \mathbf{k} \mathrm{Irr}(S) + \mathrm{Id}(S).$$
Therefore,
$f\in \mathbf{k} \mathrm{Irr}(S) + \mathrm{Id}(S).$ This is a contradiction.
\smallskip

\noindent
{\bf Case 2.} Suppose $\overline{f} \notin \mathrm{Irr(S)}$. Then the
definition of $\mathrm{Irr(S)}$ gives $\overline{f} =
q|_{\overline{s}}$, where $q\in \ncrbw^\star(\Delta X)$, $s\in S$
and $q|_s$ is normal. Then $\lbar{q|_s} = q|_{\lbar{s}} = \lbar{f}$
yielding $\overline{f-q|_s} <_n \overline{f}$. If $f=q|_s$, then $f\in
\mathrm{Id}(S)$, a contradiction. On the other hand, if $f\neq q|_s$, then $f-q|_s \neq
0$ with $\overline{f-q|_s} <_n \overline{f}$. By the minimality of
$\overline{f}$, we have
$$ f-q|_s \in \mathbf{k} \mathrm{Irr}(S) + \mathrm{Id}(S).$$
Thus
$$f\in \mathbf{k} \mathrm{Irr}(S)+ \mathrm{Id}(S),$$
still a contradiction.

Therefore $\mathbf{k} \mathrm{Irr}(S) + \mathrm{Id}(S) = \mathbf{k}
\nr$. Suppose $\mathbf{k} \mathrm{Irr}(S) \cap
\mathrm{Id}(S) \neq 0$. Let $0\neq f\in \mathbf{k}
\mathrm{Irr(S)} \cap \mathrm{Id}(S)$. Then by $f\in  \mathrm{\mathrm{Irr(S)}}$, we may write
$$
f = c_1v_1 +c_2v_2 + \cdots + c_kv_k,
$$
where $v_1 >_n v_2 >_n \cdots >_n v_k \in \mathrm{\mathrm{Irr(S)}}$. Since
$f\in \mathrm{Id}(S)$, by Item (b), we have $v_1 = \overline{f} =
q|_{\overline{s}}$ for some $q\in \nrs$, $s\in S$
and $q|_s$ is normal. This is a contradiction to the definition of
$\mathrm{\mathrm{Irr(S)}}$. Therefore $\mathbf{k} \mathrm{Irr}(S)
\oplus \mathrm{Id}(S) = \mathbf{k} \nr$ and
$\mathrm{Irr(S)}$ is a $\mathbf{k}$-basis of $\mathbf{k}\ncrbw(\Delta
X)/\mathrm{Id}(S)$.

\mref{it:cdc} $\Rightarrow \mref{it:cda}:$ Suppose $f,g\in S$ give an intersection or
inclusion composition. With the notations in the definitions of compositions, let $F=fu$ and $G=vg$ in the case of intersection composition and let $F=f$ and $G=q|_g$ in the case of
inclusion composition. Then $w:= \overline{F} =\overline{G}
$. If $(f,g)_w = F-G=0$, then we are done. If $(f,g)_w
\neq 0$, then we have
$$(f,g)_w = \sum_{i=1}^k c_i u_i,\quad 0\neq c_i\in \mathbf{k}, u_1>_n u_2>_n \cdots>_n u_k \in \nr.$$
Thus $u_i <_n \overline{F} =\lbar{G}= w$. As $(f,g)_w \in
\mathrm{Id}(S)$ and $\mathbf{k} \mathrm{Irr}(S) \cap \mathrm{Id}(S) = 0$ by Item \mref{it:cdc},
we find that $u_i$ is not in $\mathrm{Irr(S)}$ for $i=1,\cdots,k$. So by the definition of $\mathrm{Irr(S)}$, there are
$q_i\in \nrs$, $s_i\in S$ such that $u_i=
q_i|_{\lbar{s_i}}$ and $q_i|_{s_i}$ is normal for each $1\leq i\leq k$. From $\lbar{q_i|_{s_i}} = q_i|_{\lbar{s_i}} = u_i <_n w$, we have $(f,g)_w \equiv 0$ mod $[S,w]$.

Consider a composition of right multiplication $fu$ where $f\in S$, $\lbar{f} \in \nr P( \nr )$
and $u \in P( \nr )$. Then we have $fu \in \Id(S)$. By Item \mref{it:cdc}, we have
$\mathbf{k} \mathrm{Irr}(S) \cap \mathrm{Id}(S) = 0.$
By Lemma \mref{lemma:expression} this implies $fu = \sum_j d_j
q_j|_{s_j}$, where $0\neq d_j\in \bfk$, $s_j\in S$ such that $q_j\in
\nrs$, $q_j|_{s_j}$ is normal and $q_j|_{\overline{s_j}} \leq_n \overline{fu}$. Thus $fu \equiv 0$ mod $[S]$.
With a similar argument, we can show that the compositions of left multiplication are trivial $[S]$.

In summary, we have proved that $S$ is a Gr\"obner-Shirshov basis.
\end{proof}

\section{Gr\"{o}bner-Shirshov bases and free integro-differential algebras}
\mlabel{sec:gs}

We first consider a finite set $X$ and $n\geq 1$ in Section~\mref{ss:gsb} and prove that the idea $I_{\mathrm{ID},n}$ of $\bfk\nr$ possesses a Gr\"obner-Shirshov basis. Then in Section~\mref{ss:bases}, we apply the
Composition-Diamond Lemma (Theorem~\mref{thm:CD lemma}) to construct a
canonical basis for $\bfk\nr/I_{\mathrm{ID,n}}$. Letting $n$ go to infinity, we obtain a canonical basis of the free integro-differential algebra
$\bfk\calr(\dx)/I_{\mathrm{ID}}$ on the finite set $X$. For any
well-ordered set $X$, we show that the canonical basis of the free
integro-differential algebra on each finite subset of $X$ is compatible with the inclusions of the subsets of $X$ and thus obtain a canonical basis of the free integro-differential algebra on $X$.

\subsection{Gr\"obner-Shirshov basis}\mlabel{ss:gsb} In this subsection, $X$ is a finite set. Let $$ S_n:= \left\{ \phi_1(u,v), \phi_2(u,v) \mid u,v \in \nr \right\}$$
be the set of generators in Eq.~(\mref{eq:ibpn}) corresponding to the integration by parts axiom
Eq.~(\mref{eq:ibpl}). Then
$I_{\mathrm{ID,n}}$ is the differential Rota-Baxter ideal $\Id(S_n)$ of
$\bfk\nr$ generated by $S_n$.

\begin{remark} Let $u=1$. Then $\phi_1(u,v)= \phi_1(1,v)=0$ is in $S_n$. By Eqs.~(\ref{eq:diffl}) and (\ref{eq:fft}), we have
\begin{equation} \mlabel{eq:zero}
d(\phi_1(u,v)) = d(u)P(v) - d(uP(v)) +uv + \lambda d(u)v =0,  \end{equation}
and hence is in $S_n$.
Similarly, $d(\phi_2(u,v)) =0$.  So $d(S_n) \subseteq S_n$.
\end{remark}
Next, we show that $S_n$ is a Gr\"{o}bner-Shirshov basis of the differential Rota-Baxter ideal $I_{\ID,n} = \Id(S_n)\subseteq \bfk\nr$.

\begin{lemma}
  Let $u = u_0 u_1\cdots u_k \in M(\Delta X)$ with $u_0,\cdots, u_k \in \Delta
  X$. Then $\lbar{d(u)} = u_0 u_1\cdots u_{k-1}d(u_k)$.  If $u\in M(\Delta_n
  X)$, then $\lbar{d(u)} = u_0 u_1\cdots u_{k-1}d(u_k)$ provided $u_k\in
  \Delta_{n-1} X$.
\mlabel{lemma:diffleadterm}
\end{lemma}

\begin{proof}
  This follows from Eq.~(\mref{eq:prodexp}) and the definitions of the order on
  $\Delta X$.
\end{proof}

Let $\cala_{d} := \{ \lbar{ d(u)} \mid u\in S(\Delta X) \}$, $\cala_{n, d} :=  \cala_{d}  \cap M(\dnx)$  and
\begin{equation}
\nz := \{ x_0^{(n)} \cdots x_k^{(n)} \mid x_0,\cdots, x_k\in X ,k\geq 0 \}.
\mlabel{eq:nz}
\end{equation}
Note that $d(u) = 0$ for $u \in M(\dnx)$ if and only if $u = 1$ or $u\in \nz$.

\begin{lemma} We have
\begin{align*}
\left\{ \lbar{ \phi_1(u,v)  } \mid u,v\in \nr \right\} = & P(\nr \fun P(\nr)) \bigcup  \left( \bigcup_{r\geq 1} P(\nr \fun (P(\nr)\nz)^r P(\nr)) \right)\\
& \bigcup \Big(P(\big(\Lambda(\nz,\nr)\setminus P(\nr) \big)\nr) \bigcap \nr\Big) \bigcup  \{ 0\}.
\end{align*}
Here we take the intersection with $\nr$ to ensure that the right hand side is in $\nr$.
\mlabel{lemma:leadterma}
\end{lemma}

\begin{proof}
  We first show that the left hand side of the equation is contained in the
  right hand side. If $u=1$, then $\phi_1(u,v) = 0 = \lbar{\phi_1(u,v)}$.  If
  $u\in P(\nr)$, let $u=P(u_0)$ for some $u_0\in \nr$, then
$$\phi_1(u,v) = P(u_0P(v)) - P(u_0)P(v) +P(P(u_0)v) + \lambda P(u_0 v) = 0$$
and so $\lbar{\phi_1(u,v)} = 0$. Suppose that $u\neq 1$ and $u\notin P(\nr )$. Note that
$$\deg_{\dnx}( \lbar{P(d(u)P(v)) } ) = \deg_{\dnx}( \lbar{uP(v) } ) = \deg_{\dnx}( \lbar{ P(uv) } ) = \deg_{\dnx}( \lbar{P(d(u)v)} ).$$

\noindent{\bf Case 1.} $\deg_P(\lbar{d(u)}) = \deg_P(u)$. Then $$\deg_P(\lbar{ P(d(u)P(v))}) > \deg_P( \lbar{uP(v)}),  \deg_P( \lbar{P(uv)}),  \deg_P( \lbar{P(d(u)v)})$$
and so $\lbar{\phi_1(u,v)} = \lbar{ P(d(u)P(v))} = P(\lbar{d(u)P(v)}) $.
According to Eq.~(\mref{Eq:fourparts}), we have four subcases to consider. Consider first that $u = u_0P(\tu_0) \cdots u_k P(\tu_k) u_{k+1}$ with $u_0, \cdots, u_{k+1} \in S(\dnx)$ and $\tu_0, \cdots, \tu_{k+1} \in \nr$.
Since $\deg_P(\lbar{d(u)}) = \deg_P(u)$, there is at least one $u_i$ with $0\leq i\leq k+1$ such that $u_i\notin \nz$. If $u_{k+1} \notin \nz$, then $d(u_{k+1}) \neq 0$ and
$$\lbar{\phi_1(u,v)}= P(\lbar{d(u)P(v)}) =  P( u_0P(\tu_0) \cdots u_k P(\tu_k) \lbar{d( u_{k+1})} P(v) ) \in  P(\nr \fun P(\nr)).$$
If $u_{k+1}\in \nz$, suppose that $u_i$ with $0\leq i\leq k$ is right most such that $u_i\notin \nz$, then
$$\lbar{d(u)} = u_0P(\tu_0) \cdots u_{i-1}P(\tu_{i-1}) \lbar{d(u_i)}P(\tu_i) u_{i+1}P(\tu_{i+1}) \cdots u_kP(\tu_k) u_{k+1}$$
and so
$$\lbar{\phi_1(u,v)}= P(\lbar{d(u)} P(v)) \in  \cup_{r\geq 1} P(\nr \fun (P(\nr)\nz)^r P(\nr)).$$
For the other subcases, with a similar argument, we can obtain that $$\lbar{\phi_1(u,v)} \in P(\nr \fun P(\nr)) \cup  \big( \cup_{r\geq 1} P(\nr \fun (P(\nr)\nz)^r P(\nr)) \big).$$

\noindent{\bf Case 2.} $\deg_P(\lbar{d(u)}) \neq \deg_P(u)$. Then $u\in \Lambda(\nz, \nr)\setminus P(\nr)$ and $\deg_P(\lbar{ d(u) } ) = \deg_P(u) - 1$. So
$$\deg_P(\lbar{ P(d(u)P(v))}) = \deg_P( \lbar{uP(v)})= \deg_P( \lbar{P(uv)})= \deg_P( \lbar{P(d(u)v)})+1.$$
If $u\notin \nr P(\nr)$, then $\lbar{uP(v)} = uP(v)$ and $\lbar{P(uv)} = P(uv)$. By the definition of $<_n$, we have $uP(v) <_n P(uv)$.
If $u\in \nr P(\nr)$, let $u = u_0 P(u_1)$ with $u_0, u_1\in \nr$. Then by the definition of $<_n$, we have
$$ \lbar{uP(v)} = \lbar{ u_0 P(u_1)P(v)} = u_0 P( \lbar{ P(u_1) v }) <_n P( \lbar{ u_0P(u_1)v }) = \lbar{P(uv)}$$
Since $\lbar{d(u)} <_n u $, we have $\lbar{P(d(u) P(v))}, \lbar{ P(( d(u)v) }<_n
\lbar{P(uv)}$.  Hence $\lbar{\phi_1(u,v)} = \lbar{P(uv)} = P(\lbar{uv})\in
P(\Lambda(\nz,\nr) \nr)$.

We next prove the reverse inclusion. If $w = P(u_0 \lbar{ d(u_1)}
P(v)) \in P(\nr \cala_{n,d} P(\nr))$ with $u_0, v\in\nr$ and
$\lbar{d(u_1)} \in \cala_{n,d}$, let $u = u_0 u_1$. Then $\lbar{d(u)}
= u_0 \lbar{d(u_1)}$ and
$$\lbar{\phi_1(u,v)} = \lbar{P(d(u) P(v))} = P(\lbar{d(u)}P(v) ) = P( u_0 \lbar{d(u_1)} P(v)) = w. $$
If $$w = P(u_0 \lbar{ d(u_1)} u_2 P(v)) \in  \cup_{r\geq 1} P(\nr \fun (P(\nr)\nz)^r P(\nr)) $$
with $u_0, v\in \nr$, $\lbar{d(u_1)} \in \fun$ and $u_2 \in \cup_{r\geq 1}(P(\nr)\nz)^r$, let $u = u_0 u_1 u_2 $. Then $\lbar{d(u)} = u_0 \lbar{d(u_1)} u_2$ and
$$ \lbar{\phi_1(u,v)} = \lbar{P(d(u) P(v))} = P(\lbar{d(u)}P(v) ) = P( u_0 \lbar{d(u_1)} u_2 P(v)) = w. $$
If $w = P(uv) \in P(\Lambda(\nz,\nr) \nr)$ with $u\in \Lambda(\nz,\nr)$ and $v\in\nr$, then $ \lbar{\phi_1(u,v)} = P(uv) = w. $
\end{proof}

\begin{lemma} We have
\begin{align*}
&\left\{ \lbar{ \phi_2(u,v)  } \mid u,v\in \nr \right\} =  \nr \bigcap \left( P( P(\nr) \nr \fun ) \bigcup  \Big( \bigcup_{r\geq 1} P( P(\nr) \nr \fun  (P(\nr)\nz)^r ) \Big)\right. \\
&\bigcup  \Big( \left .\bigcup_{r\geq 1} P( P(\nr) \nr \fun  (P(\nr)\nz)^r P(\nr)) \Big) \bigcup P\Big( \nr (\Lambda(\nz,\nr) \setminus P(\nr) ) \Big) \right) \bigcup \{0\}.
\end{align*}
Here we take the intersection with $\nr$ to ensure that the right hand side is in $\nr$.
\mlabel{lemma:leadtermb}
\end{lemma}

\begin{proof}
The proof is similar to that of Lemma \mref{lemma:leadterma}.
\end{proof}

Note that only the first union components of Lemmas~\mref{lemma:leadterma} and \mref{lemma:leadtermb} do not involve $\nz$. Thus we have
\begin{prop}
$\{ \lbar{ \phi_1(u,v)  }, \lbar{ \phi_2(u,v) } \mid u,v \in \nr \} = P(\nr \fun P(\nr)) \cup P( P(\nr) \nr \fun) \cup \epsilon(\dnx),$
where
\begin{eqnarray*}
\epsilon(\dnx) &:=& \nr\bigcap \left( \Big( \bigcup_{r\geq 1} P(\nr \fun (P(\nr)\nz)^r P(\nr)) \Big) \bigcup P\Big( (\Lambda(\nz,\nr)\setminus P(\nr)) \nr\Big)\right. \\
&&\bigcup \Big( \bigcup_{r\geq 1} P( P(\nr) \nr \fun  (P(\nr)\nz)^r ) \Big) \\
&&
\bigcup  \Big( \left .\bigcup_{r\geq 1} P( P(\nr) \nr \fun  (P(\nr)\nz)^r P(\nr)) \Big) \bigcup P\Big( \nr (\Lambda(\nz,\nr)\setminus P(\nr)) \Big) \right) \bigcup \{0\}.
\end{eqnarray*}
\mlabel{prop:leadterm}
\end{prop}
Every term in $\epsilon(\dnx)$ has a factor in $\nz$ and will thus disappear as $n$ goes to infinity.

\begin{lemma}
The compositions of multiplication are trivial
modulo $[S_n]$. \mlabel{lemma:comtrivial1}
\end{lemma}

\begin{proof}
 Let $f \in S_n$. Then $f=\phi_1(u,v)$ or $f=\phi_2(u,v)$ for some $u,v\in \nr$. We only consider the case when
$$ f = \phi_1(u,v) = P(d(u) P(v))- uP(v)+ P(uv) + \lambda
P(d(u) v),\ u,v \in  \nr,$$
since the case for $f=\phi_2(u,v)$ is similar.
It is sufficient to show that $\phi_1(u,v)P(w)$ and $P(w) \phi_1(u,v)$ are trivial modulo $[S_n]$. We first show that $\phi_1(u,v)P(w)$ is trivial modulo $[S_n]$. Note that $\lbar{\phi_1(u,v)} \in P(\nr)$.
From Eq.~(\ref{eq:rb}) we obtain
\begin{equation}
\begin{aligned}
\phi_1(u,v) P(w) =& P(d(u)P(v))P(w) -uP(v) P(w) + P(uv)P(w) + \lambda  P(d(u) v)P(w)\\
=& P(P(d(u)P(v))w) + P(d(u)P(v)P(w)) + \lambda P(d(u)P(v)w) \\
&- u P(v)P(w) + P(uv)P(w) + \lambda  P(d(u) v) P(w) \\
=& P(P(d(u)P(v))w) + P(d(u)P(P(v)w + vP(w) + \lambda vw)) + \lambda
P(d(u)P(v)w)\\
& -uP(P(v)w) -uP(vP(w)) -\lambda u P(vw) + P(P(uv)w)+ P(uvP(w))  \\
& +\lambda P(uvw) + \lambda P(P(d(u)v)w) +\lambda P(d(u)vP(w)) +\lambda^2 P(d(u)vw).
\end{aligned}
 \label{eq5}
\end{equation}
By the definition of $\phi_1(u,v)$, we have
\begin{align}\label{eq3}
P(P(d(u)P(v))w) =P(\phi_1(u,v)w) + P(uP(v)w) - P(P(uv)w) -\lambda
P(P(d(u)v)w),
\end{align}
and
\begin{equation}
\begin{aligned} \label{eq4}
&P(d(u)P(P(v)w + vP(w) + \lambda vw))\\
=& \phi_1(u,P(v)w + vP(w) + \lambda vw ) + uP(P(v)w + vP(w) + \lambda vw) \\
& - P(u(P(v)w + vP(w) + \lambda vw)) - \lambda P(d(u) (P(v)w + vP(w) + \lambda vw)) \\
=& \phi_1(u,P(v)w + vP(w) + \lambda vw) + uP(P(v) w) + uP(vP(w)) + \lambda uP(vw) - P(u P(v) w )\\
&  - P(uvP(w)) - \lambda P(uvw) -\lambda P(d(u) P(v) w ) -\lambda
P(d(u)vP(w)) -\lambda^2 P(d(u)vw)
\end{aligned}
\end{equation}
Substituting Eqs.~(\ref{eq3}) and (\ref{eq4}) into Eq.~(\ref{eq5}), we have
\begin{align*}
\phi_1(u,v)P(w) &=  P(\phi_1(u,v)w) + \phi_1(u,P(v)w + vP(w) + \lambda vw )\\
&= P(\phi_1(u,v)w) + \phi_1(u,P(v)w)  + \phi_1(u, vP(w)) +
\lambda \phi_1(u,  vw).
\end{align*}
The last three terms are
already in $S_n$ and hence are of the form $q|_s$ with $q=\star$ and $s\in S_n$. So to show that they are trivial modulo $[S]$ we just need to bound the leading terms.

Note that
$$\lbar{P(aP(b))}, \lbar{P(P(a)b)},
\lbar{P(ab)} \leq_n \lbar{P(a)P(b)}  \text{ for } a,b\in \nr.$$
If $\deg_P(u) = \deg_P(\lbar{d(u)})$, that is, if we are in Case 1 of Lemma~\mref{lemma:leadterma}, then we have
\begin{align*}
\overline{\phi_1(u,P(v)w ) } &=  \lbar{P(d(u)P(P(v)w))} \leq_n \lbar{ P(d(u) P(v)P(w))} \leq_n \lbar{P(d(u)P(v))P(w)}= \overline{\phi_1(u,v)P(w)},\\
\overline{\phi_1(u,vP(w) ) } &=  \lbar{P(d(u)P( vP(w) ))} \leq_n \lbar{ P(d(u) P(v)P(w))} \leq_n \lbar{P(d(u)P(v))P(w)}= \overline{\phi_1(u,v)P(w)},\\
\overline{\phi_1(u, vw ) } &=  \lbar{P(d(u)P( vw ))} \leq_n \lbar{ P(d(u) P(v)P(w))} \leq_n \lbar{P(d(u)P(v))P(w)}= \overline{\phi_1(u,v)P(w)}.
\end{align*}
If $\deg_P(u) \neq \deg_P(\lbar{d(u)})$, that is, if we are in Case 2 of Lemma~\mref{lemma:leadterma}, then we have
\begin{align*}
 \overline{\phi_1(u,P(v)w ) } &=  \lbar{ P( uP(v)w )} \leq_n \lbar{P( P(uv)w )} \leq_n \lbar{ P(uv)P(w) }= \overline{\phi_1(u,v)P(w)},\\
 \overline{\phi_1(u, vP(w) ) } &=  \lbar{ P( u vP(w) )} \leq_n  \lbar{ P(uv)P(w) }= \overline{\phi_1(u,v)P(w)},\\
  \overline{\phi_1(u, vw ) } &=  \lbar{ P( u vw )} \leq_n  \lbar{ P(uv)P(w) }= \overline{\phi_1(u,v)P(w)}.
\end{align*}
Thus
$$\phi_1(u,P(v)w)  + \phi_1(u, vP(w)) + \lambda
\phi_1(u,  vw) \equiv 0 \text{ mod } [S_n,\lbar{\phi_1(u,v)P(w)}\,]$$
and so $\phi_1(u,v) P(w) \equiv 0$ mod $[S_n]$ if and only if $ P(\phi_1(u,v)w) \equiv
0$ mod $[S_n,\lbar{\phi_1(u,v)P(w)}\,]$. Let $w = w_1w_2 \cdots w_k$ be the
standard decomposition of $w$. We prove the latter statement by induction on
$\mathrm{dep}(w_1)$.

If $\mathrm{dep}(w_1) = 0$, that is, $w_1 \in M(\Delta_n X)$, let $q :=
P(\star w) \in \nrs$. Then $$q|_{\phi_1(u,v)} =P(\phi_1(u,v) w) = P(\phi_1(u,v) w_1w_2 \cdots w_k)$$ and $q|_{\phi_1(u,v)}$ is normal by $w_1\in M(\Delta_n
X)$. If $\deg_P(u) = \deg_P(\lbar{d(u)})$, then
\begin{align*}
\overline{P(\phi_1(u,v) w)} =\overline{P(P(d(u)P(v)) w)} \leq_n \overline{P(d(u)P(v)) P(w)} = \overline{ \phi_1(u,v) P(w)},
\end{align*}
If $\deg_P(u) \neq \deg_P(\lbar{d(u)})$, then
\begin{align*}
\overline{P(\phi_1(u,v) w) } =  \overline{ P( P(uv) w)} \leq_n \overline{ P(uv) P(w)} =\overline{ \phi_1(u,v) P(w)}.
\end{align*}
Hence $P(\phi_1(u,v)w)\equiv 0$ mod $[S_n]$.

If $\mathrm{dep}(w_1) > 0$, we may suppose $w_1 =
P(\tilde{w})$ with $\tilde{w} \in
\nr$. Then $w_2\in \dnx$, as $w=w_1w_2\cdots w_k$ is the standard decomposition of $w$. Since $\mathrm{dep}(\tilde{w})<\mathrm{dep}(w_1)$,
by the induction hypothesis, we may assume that
$$\phi_1(u,v) P(\tilde{w})  = \sum_i c_i
p_i|_{s_i},$$
where $0\neq c_i\in \bfk, p_i\in \nrs, s_i\in S_n$, $p_i|_{s_i}$ is normal and $\lbar{p_i|_{s_i}} \leq
\lbar{\phi_1(u,v) P(\tilde{w})}$. Let $q_i := P(p_iw_2\cdots w_k)$. Since
$p_i|_{s_i}$ is normal and $w_2\in \dnx$, it follows that
$q_i|_{s_i}$ is normal. Furthermore, we have
\begin{align*}
P(\phi_1(u,v) w) &=P(\phi_1(u,v) w_1 w_2\cdots w_k) = P( \phi_1(u,v) P(\tilde{w})w_2\cdots w_k) \\
&= \sum_i c_i P( p_i|_{s_i}w_2 \cdots w_k) = \sum_i c_i q_i|_{s_i}
\end{align*}
and
$$\lbar{q_i|_{s_i}}= \lbar{P( p_i|_{s_i} w_2\cdots w_k)} \leq_n \lbar{P( \phi_1(u,v) P(\tilde{w}) w_2\cdots w_k )} = \lbar{P(\phi_1(u,v) w)} \leq_n
\lbar{\phi_1(u,v) P(w)}.$$
Therefore $P(\phi_1(u,v)w)\equiv 0$ mod $[S_n, \lbar{\phi_1(u,v) P(w)}]$.
This completes the induction. Hence $\phi_1(u,v) P(w) \equiv 0$
mod $[S_n]$, as needed.

With a similar argument, we can show that $P(w) \phi_1(u,v) \equiv 0$ mod $[S_n]$.
\end{proof}

\begin{lemma}
There are no intersection compositions in $S_n$.
\mlabel{lemma:comtrivial3}
\end{lemma}

\begin{proof}
Let $f,g \in S_n$. By Lemmas \mref{lemma:leadterma} and \mref{lemma:leadtermb}, we have $\mathrm{bre}(\lbar{f})=1= \mathrm{bre}(\lbar{g})$. Suppose $w:= \lbar{f}u = v\lbar{g}$ gives an intersection composition.
Then by the definition of intersection composition, we have $1<\mathrm{bre}(w)<2$. This is a contradiction. Thus there are no intersection compositions in $S_n$.
\end{proof}

\begin{lemma}
The including compositions in $S_n$ are trivial.
\mlabel{lemma:comtrivial2}
\end{lemma}

\begin{proof}
We first list all possible inclusion compositions from $f, g\in S_n$, namely those $f, g\in S_n$ such that  $w:=\overline{f}=q|_{\overline{g}}$ for some $q\in \nrs$.

We begin with the case when $q=\star$. Then we have $w:=\lbar{f}=\lbar{g}$. From Lemmas \mref{lemma:leadterma} and \mref{lemma:leadtermb}, we must have
$$f =\phi_1(u,v) = g, \text{ or } f =\phi_2(u,v) = g.$$
Hence $f-g$ is trivial modulo $[S_n,w]$, as needed.

We next consider the case when $q\neq \star$. We need $\lbar{f}=q|_{\lbar{g}}$ where $\lbar{f}$ is of the form $\lbar{P(w)}$ with $w=d(u)P(v)$, $w=P(u)d(v)$ or $w=uv$ while $\lbar{g}$ is also of the form $P(d(r)P(s))$, $P(P(r)d(s))$ or $P(rs).$ Thus $q$ is of the forms
$$P(d(p)P(v)), \ P(d(u)P(p)), \ P(P(p)d(v)), \ P(P(u)d(p)), \ P(pv), \ P(up), \ P(d(u)\star), \ P(\star d(v)), $$
where $p\in \nr^\star$ and where the $\star$ in $p$ or by itself is replaced by $\lbar{g}$ which can be of the forms $P(d(r)P(s))$, $P(P(r)d(s))$ or $P(rs)$.  Thus there are
24 possibilities. The last two cases in the displayed list occur when the $P$ in $P(q)$ and the $P$ in $\lbar{g}$ coincide.
Thus all the including compositions $\lbar{f}=q|_{\lbar{g}}$ with $q\neq \star$ are of the forms
$$P(d(p|_{\lbar{g}})P(v)),  P(d(u)P(p|_{\lbar{g}})),  P(P(p|_{\lbar{g}})d(v)),  P(P(u)d(p|_{\lbar{g}})),  P(p|_{\lbar{g}}v),  P(up|_{\lbar{g}}),  P(d(u)\star|_{\lbar{g}}),  P(\star|_{\lbar{g}} d(v)), $$
with $\bar{g}= P(d(r) P(s)), P(P(r)d(s))$ or $P(rs)$.

With a similar argument as in~\cite[Lemma~5.7]{GGZ}, we can show the triviality of the ambiguities of the compositions
$$P(d(u)P(p|_{P(d(r) P(s))})), \ P(d(p|_{P(d(r) P(s))}) P(v)),\ P(d(u)P(d(r)P(s))),\  P(P(d(r)P(s)) d(v)).$$

We next check that the ambiguity of the composition
$P(d(u)P(p|_{P(P(v)d(w))}))$ is trivial.
This is the case when $w=\lbar{f}=q|_{\lbar{g}}$ where $q=P(d(u)P(p))$ for some $p\in \nr^\star$.
Then $f$ and $g$ of $S_n$ are of the form
\begin{align*}
&f = \phi_1(u,v) = P(d(u) P(v))- uP(v)+ P(uv) + \lambda P(d(u) v), \\
 &g = \phi_2(r,s)= P(P(r)d(s)) - P(r)s + P(rs) +\lambda P(rd(s)),
\end{align*}
where $\lbar{f} = \lbar{  P(d(u) P(v))}$ and $\lbar{g} = \lbar{ P(P(r)d(s))}$. Further $v=
p|_{\overline{g}}=p|_{\overline{\phi_2(r,s)}} = p|_{\lbar{ P(P(r)d(s)) } }$
for
some $p\in \nrs$ and
$$w = \overline{f} =
\overline{\phi_1(u,v)}=\lbar{P(d(u)P(v))} = \lbar{
P(d(u)P(p|_{\overline{g}})) } = \lbar{ q|_{\overline{g}}} =
q|_{\overline{g}}$$
with $q= P(d(u)P(p)) \in \nr^\star$ and $q|_g$ being normal. Then
\begin{align*}
f &=\phi_1(u,v) = P(d(u)P(p|_{P(P(r)d(s))})) -uP(p|_{P(P(r)d(s))}) +
P(u p|_{P(P(r)d(s))}) + \lambda P(d(u) p|_{P(P(r)d(s))})
\end{align*}
and
\begin{align*}
 q|_g &=
q|_{\phi_2(r,s)} = P(d(u)P(p|_{P(P(r)d(s))})) - P(d(u)P(p|_{P(r)s})) +
P(d(u)P(p|_{P(rs)})) + \lambda P(d(u)P(p|_{ P(rd(s))})).
\end{align*}
So we have
\begin{equation}
\begin{aligned}
(f,g)_w =f -q|_g =&-uP(p|_{P(P(r)d(s))}) +
P(u p|_{P(P(r)d(s))}) + \lambda P(d(u) p|_{P(P(r)d(s))}) \\
& + P(d(u)P(p|_{P(r)s})) -
P(d(u)P(p|_{P(rs)})) - \lambda P(d(u)P(p|_{ P(rd(s))})).
\end{aligned}
\mlabel{eq19}
\end{equation}
From the definition of~$\phi_1(u,v)$ and~$\phi_2(r,s)$, we have
{\small
\begin{equation}
\begin{aligned}
-uP(p|_{P(P(r)d(s))}) &=  -uP(p|_{\phi_2(r,s)}) - uP(p|_{P(r)s}) +
uP(p|_{P(rs)}) + \lambda uP(p|_{P(rd(s))}),\\
P(u p|_{P(P(r)d(s))}) &=  P(up|_{\phi_2(r,s)}) + P(up|_{P(r)s}) -
P(up|_{P(rs)}) - \lambda P(up|_{P(r d(s))}),\\
\lambda P(d(u) p|_{P(P(r)d(s))}) & = \lambda P(d(u)
p|_{\phi_2(r,s)}) + \lambda P(d(u) p|_{P(r)s}) - \lambda P(d(u)
p|_{P(rs)}) - \lambda^2 P(d(u) p|_{P(rd(s))}),\\
P(d(u)P(p|_{P(r)s})) &= \phi_1(u, p|_{P(r)s}) + uP(p|_{P(r)s}) - P(u
p|_{P(r)s}) -\lambda P(d(u) p|_{P(r)s}),\\
-  P(d(u)P(p|_{P(rs)}))&= - \phi_1(u, p|_{P(rs)}) - uP(p|_{P(rs)}) +
P(u p|_{P(rs)}) + \lambda P(d(u) p|_{P(rs)}),\\
- \lambda P(d(u)P(p|_{ P(rd(s))})) &= - \lambda \phi_1(u, p|_{P(rd(s))}) - \lambda uP(p|_{P(rd(s))}) + \lambda P(u
p|_{P(rd(s))}) + \lambda^2 P(d(u) p|_{P(rd(s))}).
\end{aligned}
\mlabel{eq20}
\end{equation}
}
From Eqs.~(\ref{eq19}) and (\ref{eq20}), it follows that
$$(f,g)_w  = -uP(p|_{\phi_2(r,s)}) + P(up|_{\phi_2(r,s)}) + \lambda
P(d(u)p|_{\phi_2(r,s)}) + \phi_1(u, p|_{P(r)s})- \phi_1(u, p|_{P(rs)}) -
\lambda \phi_1(u, p|_{P(rd(s))}).$$
By Lemma \mref{lemma:operator ideal}, we
have
$$uP(p|_{\phi_2(r,s)}), P(up|_{\phi_2(r,s)}),
P(d(u)p|_{\phi_2(r,s)}) \in \Id(S_n) $$
and
$$\phi_1(u, p|_{P(r)s}), \phi_1(u, p|_{P(rs)}),
\phi_1(u, p|_{P(rd(s))}) \in S_n\subseteq \Id(S_n).$$
Since
\begin{align*}
\overline{ uP(p|_{\phi_2(r,s)}) }, \ \overline{ P(up|_{\phi_2(r,s)})
}, \ \overline{ P(d(u)p|_{\phi_2(r,s)}) } <_n \overline{\phi_1(u,p|_{\phi_2(r,s)})}=\overline{\phi_1(u,v)} =w
\end{align*}
and
\begin{align*}
\overline{ \phi_1(u, p|_{P(r)s}) }, \ \overline{\phi_1(u, p|_{P(rs)}) }, \ \overline{ \phi_1(u, p|_{P(rd(s))}) } <_n
\overline{\phi_1(u, p|_{\overline{\phi_2(r,s)}})} =
\overline{\phi_1(u,v)}=w,
\end{align*}
we conclude that $(f,g)_w \equiv 0$ mod $[S_n,w]$.

Next, we check that the ambiguity of composition $P(P(u)d(q|_{P(d(v) P(w))}))$ is trivial. This is the case when
$w=\lbar{f}=q|_{\lbar{g}}$ for some $q=P(P(u)d(p))$ for some $p\in
\nr^\star$. Then the two elements $f$ and $g$ of $S_n$ are of the form
\begin{align*}
&f = \phi_2(u,v) = P(P(u)d(v)) - P(u)v + P(uv) +\lambda P(ud(v)), \\
&g = \phi_1(r,s)= P(d(r) P(s))- rP(s)+ P(rs) + \lambda P(d(r) s),
\end{align*}
where $\lbar{f} = \lbar{P(P(u)d(v))}$ and $\lbar{g} = \lbar{P(d(r)
  P(s))}$. Thus $v= p|_{\overline{g}}=p|_{\overline{\phi_1(r,s)}} =
p|_{\lbar{ P(d(r)P(s))} }$
for some $p\in \nrs$ and
$$w = \overline{f} =
\overline{\phi_2(u,v)}=\lbar{P(P(u)d(v))} = \lbar{
P(P(u)d(p|_{\overline{g}})) } = \lbar{ q|_{\overline{g}}} =
q|_{\overline{g}}$$
with $q= P(P(u)d(p)) \in \nrs$ and $q|_g$ being normal. Then
\begin{align*}
f &=\phi_2(u,v) = P(P(u)d(p|_{P(d(r)P(s))})) - P(u)p|_{P(d(r)P(s))} +
P(u p|_{P(d(r)P(s))}) + \lambda P(ud(p|_{P(d(r)P(s))}))
\end{align*}
and
\begin{align*}
 q|_g &=
q|_{\phi_1(r,s)} = P(P(u)d(p|_{P(d(r)P(s))})) - P(P(u)d(p|_{rP(s)})) +
P(P(u)d(p|_{P(rs)})) + \lambda P(P(u)d(p|_{ P(d(r)s)})).
\end{align*}
So we have
\begin{equation}
\begin{aligned}
(f,g)_w =&f -q|_g\\
 =& - P(u)p|_{P(d(r)P(s))} +
P(u p|_{P(d(r)P(s))}) + \lambda P(ud(p|_{P(d(r)P(s))})) \\
&  + P(P(u)d(p|_{rP(s)})) -
P(P(u)d(p|_{P(rs)})) - \lambda P(P(u)d(p|_{ P(d(r)s)})).
\end{aligned}
\mlabel{eq190}
\end{equation}
By the definition of $\phi_1(r,s)$ and $\phi_2(u,v)$, we have
{\small
\begin{equation*}
\begin{aligned}
 - P(u)p|_{P(d(r)P(s))} &=  -P(u)p|_{\phi_1(r,s)} - P(u)p|_{rP(s)} +
P(u)p|_{P(rs)} + \lambda P(u)p|_{P(d(r)s)},\\
P(u p|_{P(d(r)P(s))}) &=  P(up|_{\phi_1(r,s)}) + P(up|_{rP(s)}) -
P(up|_{P(rs)}) - \lambda P(up|_{P(d(r)s)}),\\
\lambda P(u d(p|_{P(d(r)P(s))} ) & = \lambda P(u d( p|_{\phi_1(r,s)} )) + \lambda P(u d( p|_{rP(s)} )) - \lambda P(u
d(p|_{P(rs)})) - \lambda^2 P(u d(p|_{P(d(r)s)})),\\
P(P(u)d(p|_{rP(s)})) &= \phi_2(u, p|_{rP(s)}) + P(u)p|_{rP(s)} - P(u
p|_{rP(s)}) -\lambda P(u d( p|_{rP(s)} ) ),\\
-P(P(u)d(p|_{P(rs)}))&= - \phi_2(u, p|_{P(rs)}) - P(u) p|_{P(rs)} +
P(u p|_{P(rs)}) + \lambda P(u d(p|_{P(rs)} ) ),\\
- \lambda P(P(u)d(p|_{ P(d(r)s)})) &= - \lambda \phi_2(u, p|_{ P(d(r)s)}) - \lambda P(u) p|_{ P(d(r)s)} + \lambda P(u
p|_{ P(d(r)s)}) + \lambda^2 P(u d( p|_{ P(d(r)s)} ) ).
\end{aligned}
\mlabel{eq200}
\end{equation*}
}
Then Eq.~(\ref{eq190}) becomes
$$(f,g)_w  =  -P(u)p|_{\phi_1(r,s)} + P(up|_{\phi_1(r,s)}) +  \lambda P(u d( p|_{\phi_1(r,s)} ))  + \phi_2(u, p|_{rP(s)}) - \phi_2(u, p|_{P(rs)}) - \lambda \phi_2(u, p|_{ P(d(r)s)}).$$
From Lemma \mref{lemma:operator ideal}, we have
$$ P(u)p|_{\phi_1(r,s)}, P(up|_{\phi_1(r,s)}), P(u d( p|_{\phi_1(r,s)} ))  \in \Id(S_n) $$
and
$$\phi_2(u, p|_{rP(s)}),  \phi_2(u, p|_{P(rs)}), \phi_2(u, p|_{ P(d(r)s)}) \in S_n\subseteq \Id(S_n).$$
Since
\begin{align*}
\overline{P(u)p|_{\phi_1(r,s)} }, \ \overline{ P(up|_{\phi_1(r,s)}) }, \ \overline{ P(u d( p|_{\phi_1(r,s)} ))  } <_n \overline{\phi_2(u,p|_{\phi_1(r,s)})}=\overline{\phi_2(u,v)} =w
\end{align*}
and
\begin{align*}
\overline{ \phi_2(u, p|_{ rP(s) } ) }, \ \overline{\phi_2(u, p|_{P(rs)}) }, \ \overline{ \phi_2(u, p|_{P(d(r)s )}) } <_n
\overline{\phi_2(u, p|_{\lbar{\phi_1(r,s)}})} = \overline{\phi_2(u,v)}=w,
\end{align*}
we have that $(f,g)_w \equiv 0$ mod $[S_n,w]$.

We last check the ambiguity of composition $P(p|_{P(d(r)P(s))} v )$ is
trivial.
This is the case when $w = \lbar{f} =
\lbar{q|_g}$, where $q= P(p v)$ for some $p\in \nrs$. Then $f$ and $g$
of $S_n$ are of the form
\begin{align*}
f &= \phi_1( p|_{P( d(r)P(s) ) }, v) = P(p|_{P( d(r)P(s) ) } v) + P(d(p|_{P( d(r)P(s) ) }) P(v) ) - p|_{P( d(r)P(s) ) } P(v) + \lambda P(d(p|_{P( d(r)P(s) ) }) v )\\
g &= \phi_1(r,s) = P(d(r) P(s)) - rP(s) + P(rs) - \lambda P(d(r) s),
\end{align*}
where $\lbar{f} = \lbar{P(p|_{P( d(r)P(s) ) } v)}$ and $\lbar{g} = \lbar{P(d(r) P(s)) }$. Then
\begin{equation}
\begin{aligned}
(f,g)_w =& f - q|_g\\
 = &P( d( p|_{P(d(r)P(s))} ) P( v) ) - p|_{P(d(r)P(s))} P(v) + \lambda P(d(p|_{P(d(r)P(s))}) v ) \\
& + P(p|_{rP(s)} v ) - P(p|_{P(rs)} v ) - \lambda P(p|_{P(d(r)s)} v ).
\mlabel{eq:newtri}
\end{aligned}
\end{equation}
Since
\begin{align*}
P( d( p|_{ P(d(r)P(s)) } ) P( v) ) &= P( d( p|_{ \phi_1(r,s) } ) P( v) ) + P( d( p|_{ rP(s) } ) P( v) ) - P( d( p|_{ P(rs) } ) P( v) ) - \lambda P( d( p|_{ P(d(r)s) } ) P( v) )\\
- p|_{ P(d(r)P(s)) } P(v) &= - p|_{ \phi_1(r,s) } P(v) - p|_{ rP(s) } P(v) + p|_{ P(rs) } P(v) + \lambda p|_{ P(d(r)s) } P(v)\\
\lambda P(d( p|_{ P(d(r)P(s)) } ) v )& =  \lambda P(d( p|_{ \phi_1(r,s) } ) v ) + \lambda P(d( p|_{ rP(s) } ) v ) - \lambda P(d( p|_{ P(rs) } ) v ) - \lambda^2 P(d( p|_{ P(d(r)s) } ) v )\\
P(p|_{rP(s)} v ) &= \phi_1( p|_{rP(s)}, v) - P(d( p|_{rP(s)} ) P(v) ) + p|_{rP(s)}P(v) - \lambda P(d(p|_{rP(s)}) v )\\
 - P(p|_{P(rs)} v )&= -\phi_1( p|_{P(rs)}, v) + P(d( p|_{P(rs)} ) P(v) ) - p|_{P(rs)}P(v) + \lambda P(d(p|_{P(rs)}) v )\\
 - \lambda P(p|_{P(d(r)s)} v ) &= - \lambda \phi_1( p|_{ P(d(r)s)}, v) + \lambda P(d( p|_{P(d(r)s)} ) P(v) ) - \lambda p|_{P(d(r)s)}P(v) + \lambda^2 P(d(p|_{P(d(r)s)}) v ),
\end{align*}
Eq.~(\mref{eq:newtri}) becomes
$$ f - q|_g = P( d( p|_{ \phi_1(r,s) } ) P( v) ) - p|_{ \phi_1(r,s) } P(v)+ \lambda P(d( p|_{ \phi_1(r,s) } ) v ) +  \phi_1( p|_{rP(s)}, v) -\phi_1( p|_{P(rs)}, v)  - \lambda \phi_1( p|_{ P(d(r)s)}, v).$$
From Lemma \mref{lemma:operator ideal}, we have
$$ P( d( p|_{ \phi_1(r,s) } ) P( v) ),  p|_{ \phi_1(r,s) } P(v), P(d( p|_{ \phi_1(r,s) } ) v ) \in \Id(S_n)$$
and
$$ \phi_1( p|_{rP(s)}, v), \phi_1( p|_{P(rs)}, v), \phi_1( p|_{ P(d(r)s)}, v) \in S_n \subseteq \Id(S_n).$$
Since
$$ \lbar{ P( d( p|_{ \phi_1(r,s) } ) P( v) ) },  \lbar{ p|_{ \phi_1(r,s) } P(v) },  \lbar{ P(d( p|_{ \phi_1(r,s) } ) v ) } <_n \lbar{ P( p|_{\phi_1(r,s ) } v ) } =\lbar{f} =w $$
and
$$\lbar{ \phi_1( p|_{rP(s)}, v) }, \lbar{ \phi_1( p|_{P(rs)}, v)}, \lbar{ \phi_1( p|_{ P(d(r)s)}, v)} <_n \lbar{ \phi_1( p|_{P(d(r)P(s))}, v ) } = \lbar{q|_g} = w,$$
we have that $(f,g)_w \equiv 0$ mod $[S_n, w]$.

With a similar argument, we can show the triviality of the
ambiguities of the other compositions.
\end{proof}

By Lemmas \mref{lemma:comtrivial1}, \mref{lemma:comtrivial3} and \mref{lemma:comtrivial2}, it
follows immediately that

\begin{theorem}
$S_n$ is a Gr\"{o}bner-Shirshov basis in $\bfk \nr$. Hence $\mathrm{Irr}(S_n)$ in Theorem~\mref{thm:CD lemma} is a $\bfk$-basis of $\bfk\nr /\Id(S_n)$.
 \mlabel{thm:GSbase}
\end{theorem}

\subsection{Bases for free integro-differential algebras}
\mlabel{ss:bases}
We next identify the forms of elements in $\mathrm{Irr}(S_n)$, allowing us to obtain a canonical basis of $\bfk\nr/\Id(S_n)$.

For any $u,v\in M(\dnx)$, let $u=u_1\cdots u_\ell$ and $v=v_1\cdots v_m$ with $u_i,v_j\in \dx, 1\leq i\leq \ell, 1\leq j\leq m$. Note that, by the definition of $<_n$, we have
\begin{equation*}
u <_n v \Leftrightarrow \left\{\begin{array}{l} \ell< m, \\
\text{or } \ell=m \text{ and } \exists 1\leq i_0\leq \ell \text{ such that } u_i=v_i \text{ for } 1\leq i<i_0 \text{ and } u_{i_0}<v_{i_0}, \end{array}\right.
\mlabel{eq:lex}
\end{equation*}

We now introduce the key concept to identify $\mathrm{Irr}(S_n)$.

\begin{defn}
For any $u\in M(\Delta X)$, $u$ has a unique decomposition
\begin{align}
u=u_0\cdots u_k, \text{ where } u_0,\cdots, u_k\in \Delta X. \notag
\end{align}
Call $u$ {\bf functional} if either $u = 1$ or $u_k\in
X$.
Write
$$ \cala_f:=\{u\in M(\Delta X)\,|\, u \text{ is functional }\},\ \cala_{n,f} := \cala_{f} \cap M(\Delta_n X)) \text{ and } A_f := \bfk \cala_f. $$
\end{defn}

\begin{lemma}
 $M(\Delta X) = \cala_{d} \sqcup \cala_{f}$ and $M(\Delta_n X)= \cala_{n,d} \sqcup \cala_{n,f}$.
 \mlabel{lemma:decomposition}
\end{lemma}

\begin{proof}
First we show that $\cala_{d} \cap \cala_{f} =\emptyset$. Let $\lbar{d(u)} \in \cala_{d}$ with $u\in S(\Delta X)$. Suppose
$
u=u_0\cdots u_k, \text{ where } u_0,\cdots, u_k\in
\Delta X. $
Then by Lemma \mref{lemma:diffleadterm}, we have $\lbar{ d(u)} = u_0 \cdots u_{k-1} d( u_k).$
So $\lbar{d(u)} \notin \cala_{f}$. Next we show that $M(\Delta X) = \cala_{d} \cup \cala_{f}$. Let
$u\in M(\Delta X) \setminus  \cala_{f}$. From the definition of being functional, we may suppose that
\begin{align*}
u = u_0 \cdots u_{k-1}  u_k, \text{ where } u_0,\cdots, u_{k-1}\in
\Delta X , u_k\in \dx\setminus X.
\end{align*}
Suppose $u_k = x^{(\ell)}$ for some $x\in X$ and $\ell \geq 1$. Let $v = u_0 \cdots u_{k-1}  x^{(\ell-1)}$. By Lemma
\mref{lemma:diffleadterm}, we have $u = \lbar{ d(v)} \in \cala_d$. Hence $M(\Delta X) = \cala_{d} \sqcup \cala_{f}$.

\smallskip

Since $M(\Delta_n X) \subseteq M(\Delta X)$ and $M(\Delta X) = \cala_{d} \sqcup \cala_{f}$, we have that $M(\Delta_n X)= \cala_{n,d} \sqcup \cala_{n,f}$.
\end{proof}

We now give the notion to identify the canonical basis of $\bfk \ncrbw(\Delta X)/I_{\Id}$. Write $\cala_{n,f}^0:= \cala_{n,f}\setminus \{1\}$.

\begin{defn}
Let $\fbase(\dnx)$ denote the subset of $\nr$ consisting of those $w\in \nr$ with
 \begin{enumerate}
\item if $w$ has a subword $P(u_1 u_2P(u_3))$ with $u_1, u_3\in \nr$ and $u_2\in S(\dnx)$, then $u_2$ is in $\cala_{n,f}^0$;
    \mlabel{it:bw1}
 \item if $w$ has a subword $P(P(u_1) u_2u_3)$ with $u_1, u_2\in \nr$ and $u_3\in S(\dnx)$, then $u_3$ is in $\cala_{n,f}^0$.
     \mlabel{it:bw2}
\end{enumerate}
\mlabel{de:bword}
\end{defn}

The subset $\nr$ can be defined by the following recursion based on the observation that restrictions on an element in $\fbase(\dnx)$ is imposed only to its subwords inside $P$.

For a nonempty set $Y$ and nonempty subsets $U$ and $V$ of $\frakM(Y)$, define the following subset of $\Lambda(U,V)$:
\begin{align*}
\Lambda^\prime(U,V):= &\left(\bigcup_{r\geq 0} (UP(V))^rU\right) \bigcup  \left(\bigcup_{r\geq 0} ( UP(V))^r \cala_{n,f}^0 P(V) \right) \\
&\bigcup \left(\bigcup_{r\geq 0} (P(V)U)^r P(V)\cala_{n,f}^0P(V)\right) \bigcup \left(\bigcup_{r\geq 0} (P(V)U)^rP(V) \cala_{n,f}^0\right).
\end{align*}
We define a sequence $\fbase_m:=\fbase(\Delta_n X)_m, m\geq 0,$ by taking
\begin{eqnarray*}
&\fbase_0 :=\fbase_0^\prime:= M(\dnx),
\end{eqnarray*}
and for $m\geq 0$, recursively defining
\begin{eqnarray*}
\fbase_{m+1} := \Lambda(S(\dnx),\fbase_m^\prime),\ \fbase_{m+1}^\prime := \Lambda^\prime(S(\dnx), \fbase_m^\prime).
\end{eqnarray*}
Then $\fbase_{m}$, $m\geq 0$, define an increasing sequence and we define
$$\fbase(\dnx):= \dirlim \fbase_{m} = \cup_{m\geq 0} \fbase_{m}. $$

\begin{prop}
We have
$$\mathrm{Irr}(S_n)= \fbase(\dnx) \setminus \left\{  q|_s \left | \begin{array}{ll}
q\in \calr^\star_n, s\in \epsilon(\dnx)\text{ and } q|_s \text{ is normal}\end{array} \right . \right\}.$$
\mlabel{pp:PBWBase}
\end{prop}
\begin{proof}
By Theorems \ref{thm:CD lemma} and
\ref{thm:GSbase}, we have
$$\irr(S_n)=\nr \setminus \left\{  q|_s\, \left | q\in \nrs, s\in \left\{\left .\lbar{\phi_1(u,v)},  \lbar{\phi_2(u,v)}\, \right | u, v\in \nr \right\} \text{ and } q|_s \text{ is normal}\right .\right \}.$$
By Proposition \mref{prop:leadterm}, we have
\begin{align*}
\left\{\left .\lbar{\phi_1(u,v)},  \lbar{\phi_2(u,v)}\, \right | u, v\in \nr \right\} =P(\calr_n \cala_{n,d}P(\calr_n)) \cup P( P(\calr_n) \calr_n \cala_{n,d}) \cup \epsilon(\dnx).
\end{align*}
The first and second union components correspond to restrictions imposed in items~\mref{it:bw1} and \mref{it:bw2} of Definition~\mref{de:bword} respectively.
$$
\calb(\dnx)=\nr \setminus \left \{ q|_s\,\left| q\in \nrs, s\in P(\calr_n \cala_{n,d}P(\calr_n)) \cup P( P(\calr_n) \calr_n \cala_{n,d}), q|_s \text{ is normal}\right . \right\}.
$$
Thus we have
$$\mathrm{Irr}(S_n)= \fbase(\dnx) \setminus \left\{  q|_s \, \left |\,
q\in \calr^\star_n, s\in \epsilon(\dnx) \text{ and } q|_s \text{ is normal} \right .\right\},$$
and the proposition follows.
\end{proof}

Let
\begin{equation}
S  := \left\{ \phi_1(u,v), \phi_2(u,v) \mid u,v \in \ncrbw(\Delta X)  \right\}
\mlabel{eq:gsid}
\end{equation}
be the set of generators corresponding to the integration by parts axiom Eq.~(\mref{eq:ibpl}). Then, with a similar argument to Eq.(\mref{eq:zero}), we have $d(S)\subseteq S$.

\begin{lemma}
Let $I_{\ID,n}$ $($resp. $I_{\ID}$$)$ be the differential Rota-Baxter ideal of $\bfk\nr$ $($resp. $\bfk\ncrbw(\dx)$$)$ generated by
$S_n$ $($resp. $S$$)$. Then as $\bfk$-modules we have $I_{\ID,1} \subseteq I_{\ID,2} \subseteq \cdots \subseteq I_{\ID} = \cup_{n\geq 1} I_{\ID,n}$
and $I_{\ID,n} = I_{\ID} \cap \bfk\nr$.
\mlabel{lemma:idealcompa}
\end{lemma}

\begin{proof}
Since $S_n \subseteq S_{n+1}$ and $\bfk\nr \subseteq \bfk\calr_{n+1}$ for any $n{\geq 1}$, we have $I_{\ID,1}
\subseteq I_{\ID,2} \subseteq \cdots $ and $ I_{\ID} = \cup_{n\geq 1} I_{\ID,n}$. We next show $I_{\ID,n} = I_{\ID} \cap \bfk\nr$. Obviously, $I_{\ID,n} \subseteq I_{\ID}
\cap \bfk\nr$. So we only need to verify $I_{\ID} \cap \bfk\nr \subseteq I_{\ID,n}$.
By Theorem~\mref{thm:GSbase}, we have $\bfk\nr = \bfk \irr(S_n) \oplus I_{\ID,n}$. Also $\bfk \irr(S_1) \subseteq \bfk\irr(S_2) \subseteq \cdots $. Let $n\geq 1$ and $k\geq 0$. Since $\bfk \irr(S_{n+k}) \cap I_{\ID,n+k} = 0$ and $\bfk\irr(S_{n}) \subseteq \bfk\irr(S_{n+k})$, we have $\bfk\irr(S_{n})\cap I_{\ID,n+k} = 0$. Since $I_{\ID,n}
\subseteq I_{\ID,n+k}$, by the modular law we have
\begin{equation}
I_{\ID,n+k} \cap \bfk\nr = I_{\ID,n+k} \cap (\bfk\irr(S_n) \oplus I_{\ID,n}) = (I_{\ID,n+k}\cap \bfk\irr(S_n)) \oplus I_{\ID,n} = I_{\ID,n}.
\mlabel{eq15}
\end{equation}
Let $u\in I_{\ID} \cap \bfk\nr$. By $I_{\ID} =\cup_{n\geq 1} I_{\ID,n}$, we have $u\in I_{\ID,N}$ for some $N \in \mathbb{Z}_{\geq 1}$. If $N\geq n$,
then $u \in I_{\ID,N} \cap \bfk\nr = I_{\ID,n}$ by Eq.~(\ref{eq15}). If $N<n$, then
$u\in I_{\ID,N} \subseteq I_{\ID,n}$. Hence $I_{\ID} \cap \bfk\nr \subseteq I_{\ID,n}$ and so $I_{\ID} \cap \bfk\nr = I_{\ID,n}$.
\end{proof}

Still assuming that $X$ is finite, we define
\begin{equation*}
\calr(\Delta X)_f:= \dirlim \fbase(\dnx).
\mlabel{eq:af}
\end{equation*}
Write $\cala_f^0:=\cala_f\setminus \{1\}$. Then by Definition~\mref{de:bword}, $\calr(\dx)_f\subseteq \calr(\dx)$ consists of $w\in \calr(\dx)$ with the properties that
 \begin{enumerate}
\item if $w$ has a subword $P(u_1 u_2P(u_3))$ with $u_1, u_3\in \calr(\dx)$ and $u_2\in S(\dx)$, then $u_2$ is in $\cala_{f}^0$;
 \item if $w$ has a subword $P(P(u_1) u_2u_3)$ with $u_1, u_2\in \calr(\dx)$ and $u_3\in S(\dx)$, then $u_3$ is in $\cala_{f}^0$.
\end{enumerate}

Now we have arrived at the main result of the paper.

\begin{theorem}
  Let $X$ be a nonempty well-ordered set, $\bfk\ncrbw(\dx)$ the
  free differential Rota-Baxter algebra on $X$ and $I_{\ID}$ the ideal
  of $\bfk\ncrbw(\dx)$ generated by $S$ defined in
  Eq.~(\mref{eq:gsid}). Then the composition
$$\bfk\ncrbw(\dx)_f \hookrightarrow \bfk\ncrbw(\dx) \to \bfk\ncrbw(\dx) / I_{\ID}$$
of the inclusion and the quotient map is a linear isomorphism. In other words, as $\bfk$-modules
$$\bfk\ncrbw(\dx) \cong \bfk\ncrbw(\dx)_f \oplus I_{\ID}.$$
\mlabel{thm:gsb}
\end{theorem}

\begin{proof}
  First assume that $X$ is a finite ordered set. By Theorem~\mref{thm:CD lemma} and Lemma
  \mref{lemma:idealcompa} we have
\begin{align*}
\result \cong  \bfk\nr / I_{\ID,n} =  \bfk\nr/ (I_{\ID} \cap \bfk\nr ) \cong ( \bfk\nr +I_{\ID}) / I_{\ID}
\end{align*}
From Proposition~\mref{pp:PBWBase} we have
$$\fbase(\dnx) \hookrightarrow \irr(S_{n+1}) \hookrightarrow\fbase(\Delta_{n+1}X).$$
Thus when $n$ goes to infinity, we have $\dirlim \fbase(\dnx) = \dirlim \irr(S_{n})$.  Therefore we have
\begin{align*}
\bfk \calr(\Delta X)_f = \dirlim (\bfk\fbase(\dnx))  = \dirlim (\result) \cong \dirlim ((\bfk\nr + I_{\ID})/I_{\ID}) = \bfk\ncrbw(\dx) / I_{\ID},
\end{align*}
since $\dirlim \nr = \calr(\dx).$

Now let $X$ be a given nonempty well-ordered set and $u\in
\bfk\ncrbw(\dx)$. Then there is a finite {\it ordered} subset
$Y\subseteq X$ such that $u$ is in $\bfk\ncrbw(\Delta Y)$.  Then by
the case of finite sets proved above, $u\in \bfk\ncrbw(\Delta Y)_f +
I_{Y,\ID}$.
By definition, we have $\bfk\ncrbw(\Delta Y)_f \subseteq
\bfk\ncrbw(\Delta X)_f$ and $I_{Y,\ID} \subseteq I_{\ID}$.  Hence
$u\in \bfk\ncrbw(\Delta X)_f + I_{\ID}$. This proves
$\bfk\ncrbw(\Delta X) = \bfk\ncrbw(\Delta X)_f + I_{\ID}$.

Further, if $0\neq u$ is in $I_{\ID}$,
then there is a finite ordered subset $Y\subseteq X$ such that $u$ is in $I_{Y,\ID}$. Thus $u \notin \bfk\ncrbw(\Delta Y)_f$ since $\bfk\ncrbw(\Delta Y)_f \cap I_{Y,\ID}=0$.
By the definition of $\bfk\ncrbw(\Delta X)_f$, we have $\bfk\ncrbw(\Delta Y)\cap \bfk\ncrbw(\Delta X)_f = \bfk\ncrbw(\Delta Y)_f$.
Therefore $u \notin \bfk\ncrbw(\Delta X)_f$. This proves $\bfk\ncrbw(\Delta X)= \bfk\ncrbw(\Delta X)_f \oplus I_{X,\ID}$\,.
\end{proof}
\smallskip

\noindent
{\bf Acknowledgements}:
This work was supported by the National Natural Science Foundation of China (Grant No. 11201201, 11371177 and 11371178), Fundamental Research Funds for the Central Universities (Grant No. lzujbky-2013-8),
the Natural Science Foundation of Gansu Province (Grant No. 1308RJZA112), the National Science Foundation of US (Grant No. DMS~1001855) and the
Engineering and Physical Sciences Research Council of UK (Grant No. EP/I037474/1).

\end{document}